\documentclass[11pt]{amsart}
 
\usepackage{amsmath,amsthm,amssymb,graphicx,mathtools,hyperref,bm,verbatim,color,soul,amsfonts,pifont,centernot,bbm,float,dsfont,mathrsfs,tcolorbox,comment,enumitem} 
\usepackage[margin=1in]{geometry} 
\usepackage[utf8]{inputenc} 
\usepackage[english]{babel} 
\usepackage{quotes}

\allowdisplaybreaks
\hypersetup{colorlinks,linkcolor=blue}

\newcommand{\n}{\mathbb{N}}

\newcommand{\real}{\mathbb{R}}
\newcommand{\R}{\mathbb{R}}
\newcommand{\bra}{\langle}
\newcommand{\ket}{\rangle}

\newcommand{\bound}[1]{\partial{#1}}
\newcommand{\mbound}[1]{m(\partial{#1})}
\newcommand{\inti}[1]{\E({#1})}
\newcommand{\minti}[1]{m\big(\E({#1})\big)}

\newcommand{\rayleigh}[1]{ \frac {\bra \Delta {#1},{#1} \ket} {\bra {#1},{#1} \ket} }

\newcommand{\compl}{\mathsf{c}}
\newcommand{\h}{\overline{h}}
\newcommand{\supp}{{\sf supp}}
\renewcommand{\k}{{\sf k}}

\DeclareMathOperator{\V}{V}
\DeclareMathOperator{\E}{E}

\DeclareMathOperator{\VV}{\mathcal{V}}
\DeclareMathOperator{\Prob}{Prob}

\DeclareMathOperator{\I}{I}

\theoremstyle{definition}
	\newtheorem{definition}{Definition}
	
\theoremstyle{plain}
	\newtheorem{theorem}[definition]{Theorem}
	\newtheorem{lemma}[definition]{Lemma}
	\newtheorem{corollary}[definition]{Corollary}
	\newtheorem{proposition}[definition]{Proposition}

\theoremstyle{remark}
	\newtheorem{remark}[definition]{Remark}
	\newtheorem{example}[definition]{Example}
	\newtheorem{question}[definition]{Question}
	\newtheorem{problem}[definition]{Problem}
	
	\newtheorem{notation}[definition]{Notation}

\title{Spectra of infinite graphs with summable weight functions}
\author{Michael Farber}
\address{School of Mathematical Sciences\\
Queen Mary University of London\\ E1 4NS London\\UK.}
\email{m.farber@qmul.ac.uk}
\author{ Lewin Strauss}
	\address{School of Mathematical Sciences\\
Queen Mary University of London\\ E1 4NS London\\UK.}
\email{lewin.strauss@gmail.com}
	
	\thanks{Both authors were partially supported by a grant from the Leverhulme Trust}

\begin{document}
\begin{abstract} In this paper we study spectra of Laplacians of infinite weighted graphs. Instead of the assumption of local finiteness we impose the condition of summability of the weight function. 
Such graphs correspond to reversible Markov chains with countable state spaces. 
We adopt the concept of the Cheeger constant to this setting and prove an analogue of the Cheeger inequality characterising the spectral gap. 
We also analyse the concept of the dual Cheeger constant originally introduced in \cite{B14}, which allows estimating 
the top of the spectrum. In this paper we also introduce a new combinatorial invariant, $\k(G,m)$, which allows a complete characterisation of bipartite graphs and measures the asymmetry of the spectrum (the Hausdorff distance between the spectrum and its reflection at point $1\in \Bbb R$). We compare $\k(G, m)$ to the Cheeger and the dual Cheeger constants. 
Finally, we analyse in full detail a class of infinite complete graphs and their spectra.
\end{abstract}

\maketitle

\noindent
Data sharing is not applicable to this article as no datasets were generated or analysed during the current study.

\section{Introduction}

Spectral graph theory is a well-developed field of mathematics lying at the crossroads between combinatorics, analysis and geometry. It has multiple applications in engineering and computer science \cite{Chu}. The celebrated {\em Cheeger inequalities} relate the geometry of a graph with the spectral properties of the graph Laplacian \cite{D84}. The Cheeger inequalities 
play a crucial role in theory of expander graphs \cite{HLW06}, which are sequences of finite graphs with certain asymptotic properties. The literature on spectral graph theory contains many important results about finite and infinite graphs and their spectral properties, see for instance \cite{Chu, Keller, Wo}.

In this paper, we consider countable weighted graphs that are not necessarily locally finite but
we impose an assumption of summability on the weights, i.e. we require the sum of all edge weights to be finite. 
Such {\em summable weighted graphs} can be thought of as representations of reversible countable Markov chains \cite{Wo}.

We introduce three geometric constants and analyse their bearing on the spectral properties of the normalised graph Laplacian of summable weighted graphs. The Cheeger constant and the dual Cheeger constant introduced in this paper can be compared to the invariants studied in \cite{B14}. There are several significant differences between \cite{B14} and our approach: (a) in \cite{B14} the authors study a Dirichlet type Laplace operator and their approach is applicable to locally finite graphs only; 
(b) The Cheeger constant defined in \cite{B14} measures bottom of the spectrum which is automatically $0$ for summable weighted graphs studied in this paper. Our Cheeger constant measures the spectral gap rather than bottom of the spectrum.

Our new combinatorial invariant, $\k(G, m)$, affords a complete characterisation of bipartite graphs: it vanishes if and only if the graph is bipartite. We also show that the new invariant $\k(G, m)$ allows estimating the asymmetry of the spectrum (Theorem \ref{thm:mains}). 

In section \ref{sec:5} we show several examples of graphs $G$ such that different summable weight functions on $G$ have either vanishing or non-vanishing spectral gap illustrating the fact that the property to have a 
non-vanishing spectral gap (i.e. \lq\lq expansion\rq\rq) strongly depends on the weight function (and not just on the underlying graph). 

In more detail, our main results can be summarised as follows.
Let $(G,m)$ denote a connected summable weighted graph, and let $\sigma(\Delta) \subset [0,2]$ denote the spectrum of the associated Laplacian $\Delta$.
Summability implies $0 \in \sigma(\Delta)$.  
We provide geometric estimates for the spectral gap and top of the spectrum,
\[ \lambda_{gap}(\Delta) = \inf\{\sigma(\Delta)-\{0\}\}\quad
\text{ and } \quad
\lambda_{top} = \sup \{\sigma(\Delta)\}. \]
Namely, we define a {Cheeger constant} $h(G,m)$ and a {dual Cheeger constant} $\overline{h}(G,m)$, and we prove (cf. Theorem \ref{thm:ci})
\begin{align*}
	1 - \sqrt{ 1 - h(G,m)^2} \, \leq \, &\lambda_{gap}(\Delta) \leq 2 \cdot h(G,m), \\
 2 \cdot \overline{h}(G,m) \, \leq \, &\lambda_{top}(\Delta) \leq 1 + \sqrt{1 - \big(1 - \overline{h}(G,m)\big)^2 }.
 \end{align*}
Moreover, we define a new geometric constant ${\k}(G,m)$ and show that the Hausdorff distance between 
$\sigma(\Delta) \subset [0,2]$ and its reflection around 1 is bounded above by $2 \cdot {\k}(G,m)$ (Theorem \ref{thm:mains}). 

Finally, in sections \ref{sec8}, \ref{sec9}, and \ref{sec10}, we analyse a rich class of infinite complete graphs whose spectra admit a particularly detailed description.

We are not in a position to survey the vast literature which pertains to various classes of graphs, 
various Cheeger-type combinatorial constants, various graph Laplacians, and various aspects of the spectra of these Laplacians. As far as we know the results contained in our paper are new and are not contained in any previously published articles. 

We may mention  \cite{JEMS15} where the authors use the concept of intrinsic metrics and develop a comprehensive framework for countable weighted graphs, of which our model is a special case. The authors introduce a Cheeger-type constant (distinct from ours) and use it to bound the bottom of the spectrum of the graph Laplacian. Our assumption of summability implies that the bottom of the spectrum is $0$. 
As another example we may mention Theorem  3.5 of \cite{MS03} which provides  a  lower  bound  of  the  spectral  gap of a normalised Laplacian, but in \cite{MS03} the  underlying  graphs  are  implicitly  assumed  to  be locally finite as follows  from Definition 2.2 of \cite{MS03}.\footnote{For any oriented edge $vw$, the ratio $i(wv) / i(vw)$ is bounded if and only if $\mu(v)/\mu(w)$ is bounded. But $\mu$ is assumed to be summable over vertices, so the vertex $v$ can only have finitely many adjacencies.}

The authors thank Norbert Peyerimhoff for helpful advice.

\section{Summable weighted graphs} \label{sec:prelims} 

\subsection{Definitions}
A graph is a 1-dimensional simplicial complex. For a graph $G$, we denote by $\V(G)$ and $\E(G)$ the sets of vertexes and edges, respectively. We say that $G$ is countable if the vertex set $\V(G)$ is either finite or countably infinite. 
{\it A weight function} on $G$ is a function $m: \E(G)\to (0,\infty)$. {\it A weighted graph} is a pair $(G, m)$ consisting of a graph $G$ and a weight function $m$ on $G$. 

\begin{definition} We shall say that a countable weighted graph $(G, m)$ is {\it summable} if the sum of all edge weights is finite, $\sum_{e \in \E(G)} m(e) < \infty$. 
\end{definition}

The weight function of a summable weighted graph $(G, m)$ naturally extends to the vertexes: we set 
$m(v)=\sum_{v \in e} m(e).$ 
In other words, the weight of a vertex is the sum of the edge weights over all edges that are incident to it ("weighted degree"). 
According to this definition, a vertex has weight $0$ iff it is isolated. 

Below we shall consider only graphs without isolated vertexes; we shall have $m(v)>0$ for any vertex $v$. 
The resulting function $m: \V(G)\to [0, \infty)$ defines a $\sigma$-additive measure on $\V(G)$. The weight 
$m(S)$ of a subset $S\subset \V(G)$ is defined as $\sum_{v\in S} m(v)$, the sum of the weights of all elements of $S$. Note that 
$$\sum_{v\in \V(G)} m(v) = 2\cdot \sum_{e\in \E(G)} m(e) < \infty.$$

We shall consider the Hilbert space  $L^2(G,m)$ 
of square integrable functions 
$f\colon \V(G) \rightarrow \real$
with respect to $m$. The elements $f\in L^2(G,m)$
 satisfy
		$ \sum_{v\in \V(G)} m(v) \cdot f(v)^2 < \infty. $
The inner product of $L^2(G,m)$ is given by
$$ \langle f,g\rangle = \sum_{v\in \V(G)} m(v) \cdot f(v) \cdot g(v).$$
Note that any constant function is square integrable, i.e. constant functions belong to $L^2(G,m)$.

The {\em normalised Laplacian} 
of a summable weighted graph $(G,m)$ without isolated vertexes is defined by
\begin{eqnarray}\label{def:lap}
\Delta f(v) = f(v) - \sum_{w\sim v} \frac{m(vw)}{m(v)} \cdot f(w), \quad f\in L^2(G, m).
\end{eqnarray}
Using the Cauchy-Schwarz inequality, one sees that the sum in (\ref{def:lap}) converges for any $f\in L^2(G, m)$. More precisely, for any $f\in L^2(G, m)$, 
\begin{eqnarray}\label{ineq111}  \sum_w m(vw)f(w) 
\le 
\left[\sum_w m(w)\left(\frac{m(vw)}{m(w)}\right)^2\right]^{1/2}\cdot ||f||\le m(v)^{1/2}\cdot ||f||,\end{eqnarray}
where $$||f||^2 = \langle f, f\rangle=\sum_v m(v)f(v)^2.$$
Note the following formula:
\begin{equation} \label{eq:delhelp}
		\bra \Delta f, f \ket = \sum_{vw \in \E(G)} m(vw) \cdot \big( f(v) - f(w)\big)^2 .
		 \end{equation}

\begin{lemma} \label{lem:prop}
For a summable weighted graph $(G, m)$, the Laplacian $\Delta\colon L^2(G,m) \rightarrow L^2(G,m)$ is well-defined; it is self-adjoint, non-negative, and bounded. Moreover, the spectrum $\sigma(\Delta)$ of the Laplacian lies in $[0,2]$.
Any constant function $f: \V(G)\to \R$ satisfies $\Delta f=0$, and thus $0\in \sigma(\Delta)$ is an eigenvalue.
\end{lemma}

\begin{proof}
We have $\Delta = I-P$, where
\begin{eqnarray}\label{pp}
P(f)(v) = m(v)^{-1} \sum_w m(vw)f(w).
\end{eqnarray} 
Clearly $P$ is self adjoint.
Using (\ref{ineq111}), one sees that 
$||P(f)|| \le ||f||$, which implies that the spectrum of $P$ lies in $[-1,1]$. Therefore $\sigma(\Delta)\subset [0, 2]$. 
\end{proof}
	
Clearly, if the graph $G$ is infinite not every point of the spectrum $\sigma(\Delta)$ is an eigenvalue. 

\subsection{Spectral gap}

We shall be interested in the {\em spectral gap} of $\Delta$, defined by $ \lambda_{gap}(\Delta) = \inf \{ \lambda \in \sigma(\Delta)\colon \lambda > 0 \}.$ 
The spectral gap can be characterised as follows:
\begin{equation} \label{eq:char}
					 \lambda_{gap} = \inf \bigg\{ \rayleigh{f} \colon f \in L^2(G,m), \ f \perp {\bf 1} \bigg\},  \end{equation}
 see \cite{RS}, Chapter 13. Here ${\bf 1}\colon \V(G)\to \R$ is the constant function equal to $1$ at all points.

\begin{lemma} If $(G, m)$ is a summable weighted graph such that the underlying graph $G$ is either infinite or it is finite but not a complete graph, then $\lambda_{gap}\le 1$. 
\end{lemma}

\begin{proof} For a couple of vertexes $a, \, b\in \V(G)$ define $f_{ab} \in L^2(G,m)$, 
 via $f(a) = m(b)$, $f(b)=-m(a)$ and $f(v)=0$ for $v\notin \{a, b\}$. Then $f_{ab}\perp \bf 1$, and 
 using formula (\ref{eq:delhelp}) we find
    \[ \rayleigh{f_{ab}} = 1 +2\cdot \frac {m(ab)}{ m(a) + m(b)}.\]
    If $G$ is not a complete graph we can select the vertices $a, b$ such that $m(ab)=0$ (i.e. the edge connecting $a$ and $b$ is not in $G$); then $\lambda_{gap}\le 1$ by (\ref{eq:char}). 
    
If $G$ is complete and infinite then we can choose a sequence $b_i$ of vertices that satisfies 
$m(ab_i) \to 0$; such sequence exists since $G$ is summable. Then the sequence $$\rayleigh{f_{ab_i}}$$ converges to 1 implying $\lambda_{gap}\le 1$ by (\ref{eq:char}).
\end{proof}

\begin{remark}
There are examples of graphs with spectral gap greater than 1: for a complete graph on $n$ vertices weighted by $m(e) = 1$ for all $e \in \E(G)$, the spectral gap equals $\frac{n}{n-1}>1$ (see \cite{Chu},  pg. 6). 
\end{remark}

%
%
%

\subsection{Summable weighted graphs as reversible Markov chains} \label{sec:markov} 
A weighted graph $(G, m)$ with summable weight function $m$ determines a Markov chain with the state space $\V(G)$, where the probability of moving from $v$ to $w$ equals $$p_{vw}=\frac{m(vw)}{m(v)}.$$ 
As above, we assume that $G$ has no isolated vertexes. If we write 
$M= \sum_v m(v) = 2\sum_e m(e)$, then the function 
$v\mapsto \phi(v) = m(v)\cdot M^{-1}$ is a stationary probability distribution on $\V(G)$. 
The Markov chains arising from summable weighted graphs are reversible and recurrent, see \cite{Wo}.

\section{The Cheeger constant of a summable weighted graph and its dependence on the weight function}\label{sec:5}

\subsection{Definition of the Cheeger constant} \label{subsec:def} 
        
        The Cheeger constant is a real number between 0 and 1 that, intuitively speaking, measures the robustness of a connected weighted graph $(G, m)$.

  Let $S$ be a set of vertices in $G$.
        	 The {\em boundary} of $S$, denoted $\partial S$, is the set of all edges in $G$ with one endpoint in $S$,
        			\[ \partial S = \{ e \in \E(G) \colon |e \cap S| = 1 \}. \]
        		 The {\em interior} of $S$, denoted $\I(S)$, is the set of all edges in $G$ with both endpoints in $S$,
        			\[ \I(S) = \{ e \in \E(G) \colon |e \cap S| = 2\}.\]
        			
      We shall denote by $S^c$ the complement, $S^c=\V(G)-S$. Besides, $m(S)$ stands for $\sum_{v\in S} m(v)$ and $m(\partial S)$ denotes $\sum_{e\in \partial S} m(e)$.  These entities are related via
      \begin{eqnarray}\label{eq:weights}
      m(S) = m(\partial S) + 2\cdot m(\I(S)).
      \end{eqnarray}	 
      	        
        \begin{definition} \label{def:iso}
       	 	Let $S$ be a non-empty set of vertices in $G$.
		\begin{itemize}
        			\item[(a)] The {\em Cheeger ratio} of $S$, denoted $h(S)$, is the number
        			\[ h(S) = \frac{m(\partial S)}{m(S) }. \]
        			\item[(b)] The {\em Cheeger constant} of $(G,m)$, denoted $h(G,m)$, is the number
        			\[ h(G,m) = \inf \big\{ h(S) \big\},\]
        			where the infimum is taken over all non-empty sets of vertices $S$ that satisfy 
        			\begin{equation} \label{eq:half}
        				m(S) \leq m({S}^\compl).
        				\end{equation}
			\end{itemize}
        			
       	 	\end{definition} 
It follows from Equation (\ref{eq:weights}) that 
       $ h(G,m) \in [0,1].$

        We consider the Cheeger constants of some interesting weighted graphs in Subsection \ref{subsec:examples}.
        
	

\subsection{Cheeger sets} \label{subsec:cheegsets} 
A set of vertexes $S\subset \V(G)$ satisfying (\ref{eq:half}) is a Cheeger set if the induced subgraph $G_S$ is connected. The collection of {\em Cheeger sets} in $(G,m)$ be denoted $\VV(G,m)$.
 \begin{lemma} \label{lem:prac}
        		The Cheeger constant $h(G, m)$ equals the infimum of the Cheeger ratios taken over all Cheeger sets:
        		\[ h(G,m) = \inf \big\{ h(S)\colon  S \in \VV(G,m)\big\}.\]
\end{lemma}
%
%
%
	
\begin{proof}
		Let $S\subset \V(G)$ be a non-empty subset that satisfies (\ref{eq:half}).
            	We may enumerate the connected components of the induced subgraph $G_S$ and denote by $S_i$ the vertex set of the $i$-th component. Then for $i\not=j$, one has $\partial S_i \cap \partial S_j =\emptyset$ and $m(\partial S) = \sum_i m(\partial S_i)$. 
	We obtain
            	\[ h(S) = \frac{ \sum_i m(\partial S_i) }{ \sum_i m(S_i) } \geq \inf_i\left\{\frac{m(\partial S_i)}{m(S_i}\right\}
=
\inf_i\big\{ h(S_i)\big\}. \]
            	Since $S_i$ is a Cheeger set for all $i$, the result follows from Definition \ref{def:iso}(b).
		\end{proof}

\subsection{Examples} \label{subsec:examples} 

	The Cheeger constant of a weighted graph depends not only on the structure of the underlying graph, but also on its weight function.
	In this subsection, we we consider two structurally very different graphs and equip each of them with two different summable weight functions.
	Remarkably, both graphs exhibit a vanishing Cheeger constant in one case and a large Cheeger constant in the other.        		
		

%
        			
        \begin{example} [The infinite complete graph with weight function $m_1$] \label{ex:comp1} 
        Denote by $K$ the infinite complete graph with vertex set $\V(K) = \n$. We show below that different weight functions on $K$ can lead to vastly different Cheeger constants.

            	Let $(p_i)_{i \in \n}$ be a sequence of positive real numbers that sum to one,
		$\sum_{i \in \n} p_i = 1.$		
            	We define a weight function $m_1$ on the infinite complete graph $K$ by
            	\[ m_1(ij) = p_i \cdot p_j , \quad i, j\in \n.\]
	We have $\sum_{ij} m(ij)=1$, i.e. this weight function is summable. Besides, 
    $m_1(i) = p_i - p_i^2$ and therefore $m_1(\n)=1- \sum_i p_i^2$. 
		Every Cheeger set $S \in \VV(K,m_1)$ satisfies
            	\begin{align*}
            		h(S) 				
            		= \frac{ \sum_{i\in S} p_i \cdot \sum_{j\not \in S} p_j}{m_1(S)} 
            		>\frac{ m_1(S) \cdot m_1({S}^\compl) } {m_1(S)} 
            		\geq \frac{m_1(\n)}{2}.
            		\end{align*}
            	Therefore, for the Cheeger constant one has
            	\[ h(K, m_1)  \geq \frac{m_1(\n)}{2} .\]
            	\end{example}
        
        Example \ref{ex:comp1} shows that we can equip the infinite complete graph with a summable weight function such that the Cheeger constant of the resulting weighted graph is relatively large.
        Whilst it is tempting to attribute this observation solely to the robustness of complete graphs, the following example suggests otherwise. In section \ref{sec8} we shall describe the spectrum of the weighted graph $(K, m_1)$ in more detail; in particular, we shall see that the spectral gap equals 1, cf. Proposition \ref{prop:spectrumd}. 
        		
        \begin{example}[The infinite complete graph with weight function $m_2$] \label{ex:comp2}
            	Now we define a different weight function $m_2$ on the infinite complete graph $K$ via
            	\[ m_2(ij) = 
            	\begin{cases} 
               		\frac{1}{j^2}& \mbox{ if } |j-i| = 1, \\
              		\frac{1}{j!} & \mbox{ if } |j-i| > 1, 
              		\end{cases} 
            		\quad \text{where } j > i. \]
		The weight function $m_2$ is summable:
		\[ \sum_{ij} m(ij) = \sum_{i = 1}^\infty \bigg( \frac{1}{i^2} + \frac{i -1}{i!} \bigg) < \infty .\]
            	Write $T_n =  \{ n, n+1, \ldots \}$. 
            	The boundary and the interior of $T_n$ satisfy
            	\begin{align*}
            		m_2(\partial T_n) &= \frac{n-1}{n!} + \frac{1}{n^2} +  n\cdot \sum_{i = n + 1}^\infty \frac{1}{i!},  \\
            		m_2\big(\I(T_n)\big) &= \sum_{ i = n+2}^\infty \frac{i-n-1}{i!} +  \sum_{ i = n+1}^\infty \frac{1}{i^{2}} \ge \frac{1}{4n} . 
            		 \end{align*}
            	With regard to the third summand in the first expression, we observe
            	\[ n\cdot \sum_{i = n + 1}^\infty \frac{1}{i!} < n \cdot \frac{1}{n!} \cdot \sum_{i = 1}^\infty \frac{1}{(n+1)^i} = \frac{1}{n!}. \]
            	It follows that, for large $n$, the boundary of $T_n$ satisfies
            	$m_2(\partial T_n) < \frac{3}{n^2}.$
            	Therefore, 
            	\[ \frac{ m_2\big(\I(T_n)\big)}{m_2(\partial T_n)}  >  \frac{n^2}{3} \cdot \frac{1}{4n}\to \infty.\]
            	Since 
	\begin{eqnarray}\label{eq:hinv}
	h(T_n)^{-1} -1 = 2\cdot \frac{m(\I(T_n)}{m(\partial T_n))},\end{eqnarray}
	it follows that the Cheeger constant of the weighted graph $(K, m_2)$ vanishes, 
		$ h(K, m_2) = 0 .$
            	\end{example}
        
     Examples \ref{ex:comp1} and \ref{ex:comp2} illustrate the fact that the Cheeger constant $h(K,m)$ strongly depends on the weight function $m$. Below we analyse more examples of this kind. 
%
  
        

        						
        \begin{example}[The Half-Line graph with weight function $m_3$] \label{ex:line1}
        
         The Half-Line graph, denoted by $H$, comprises the vertex set $\V(H) = \n$ and edges of the form 
        $e_i = \{i-1,i\}$ for all $i \ge 1$.
            	We define a weight function $m_3$ on the $H$ via
            	$ m_3(e_i) = i^{-2} .$
		The weight function $m_3$ is summable as the series $\sum_{i\ge1} i^{-2}$ converges. 
		We show below that $h(H, m_3)=0$. 
		
		Write $T_n =  \{ n, n+1,  \ldots  \}\subset \V(H)$. 
		The boundary and interior of $T_n$ satisfy
            	\begin{align*}
            		m_3(\partial{T_n)} = n^{-2}, \quad
            		m_3\big(\I({T_n})\big) = \sum_{i = n+1}^\infty i^{-2}.
            		\end{align*}
            	Therefore, 
		\[ \frac{m_3\big(\I({T_n})\big)}{m_3(\partial{T_n)}}  = n^2 \cdot \sum_{i = n+1}^\infty i^{-2} \rightarrow \infty ,\]
		which, by (\ref{eq:hinv}), gives $h(T_n)\to 0$, implying 
       $ h(H, m_3) = 0 .$				
            	\end{example}	
%
	
        \begin{example}[The Half-Line graph with weight function $m_4$]\label{ex:line2}
            	Here we define another weight function on the Half-Line graph $H$:
            	$m_4(e_i)= r^i,$ where $ r \in (0,1). $ We show below that the Cheeger constant $h(H, m_4)>0$. 
		
            	As before, write $T_n =  \{ n, n+1, \ldots  \}$. 
            	For $n>0$, one has $m_4(\partial T_n) = r^n$ and $m_4(T_n) = r^n \cdot \frac{1+r}{1-r}$ and therefore $h(T_n)= \frac{1-r}{1+r}$ is independent of $n$. Note that inequality (\ref{eq:half}) is satisfied for any $n$ large enough. 
	We want to show that 
	\begin{eqnarray}\label{ex:line4}
	h(H, m_4)= h(T_n)= \frac{1-r}{1+r}.\end{eqnarray}
	By Lemma \ref{lem:prac} we need to consider subsets $S\subset \V(H)$ such that the induced subgraphs 
	$H_S$ are connected; in other words $S$ must be an interval, finite or infinite. 
Let $S = \{ i, \ldots, j\}$ be a finite interval. Then 
$$m_4(\partial S) = m_4(\partial T_i)+r^{j-1}, \quad m_4(S) = m_4(T_i)-m_4(T_{j+1}),$$
and therefore 
$$h(S)= \frac{m_4(\partial S)}{m_4(S)} = \frac{m_4(\partial T_i) +r^{j-1}}{m_4(T_i)-m_4(T_{j+1})} > h(T_i)
=\frac{1-r}{1+r}.$$
This proves (\ref{ex:line4}). 
\end{example}

%
%

\section{The dual Cheeger constant}   \label{sec:dualcheeger} 

In \cite{B14}, the authors introduced a new geometric constant, which they call the dual Cheeger constant. 
The dual Cheeger inequalities state that this constant controls the top of the spectrum of the Laplacian.
The construction in \cite{B14} is restricted to locally finite weighted graphs.

The purpose of this section is
to introduce the notion of a dual Cheeger constant adopted for weighted graphs with summable weight functions.


\subsection{Definition of the dual Cheeger constant} \label{subsec:def2}

For all $A,B \subset \V(G)$, denote by $\E(A,B)$ the set of all edges connecting $A$ to $B$.
The symbol $m(A,B)$ will denote  $m(\E(A,B)).$

\begin{definition} \label{def:dualiso}
For $A, B\subset \V(G)$ with $A\cap B=\emptyset$, $A\not=\emptyset\not= B$, write 
\begin{eqnarray}\label{dualch}
\overline{h}(A,B) = \frac{2 \cdot m(A,B)}{m(A) + m(B) }.
\end{eqnarray} 
The {\em dual Cheeger constant} of $(G,m)$, denoted by $\overline{h}(G,m)$, is defined as
    			\[ \overline{h}(G,m) = \sup \big\{ \overline{h}(A,B) \big\},\]
    			where the supremum is taken over all disjoint nonempty sets of vertices $A, B$ in $G$.
	\end{definition} 
Since $m(A)\ge m(A, B)$ and $m(B)\ge m(A, B)$, we see that $\overline h(A, B) \le 1$, and therefore 
$$\overline h(G, m) \le 1$$ for any weighted graph $(G, m)$. 

If the graph $G$ is bipartite and $V(G)=A\sqcup B$ is a partition of the set of vertexes such that all the 
edges 
connect $A$ to $B$, then $m(A)= m(A, B)$ and $m(B)= m(A, B)$, which implies
$\overline h(A, B)=1$, and therefore 
\begin{eqnarray}\label{bip}
\overline h(G, m)=1\end{eqnarray}
for any bipartite $(G, m)$. 

The inequality $\overline h(G, m)<c$ is equivalent to the statement that, for any pair of disjoint subsets $A, B\subset \V(G)$, one has the inequality
\begin{eqnarray}
m(A, B) \le c\cdot \frac{m(A)+m(B)}{2}.
\end{eqnarray}
In other words, the weight of the connecting edges between any pair of disjoint subsets $A$ and $B$ is at most $c$ times the average weight of $A$ and $B$. 
%

In \cite{B14}, the authors consider {locally finite} weighted graphs $(G,m)$ whose weight function is {not necessarily summable}.
Given an exhaustion $\Omega_n \uparrow \V(G)$ (a filtration of connected subsets that converges to $\V(G)$), they write
\[ \overline{h}(\Omega_n)= \sup \bigg\{ \frac{2 \cdot m(A,B)}{m(A) + m(B)} \bigg\},\]
where the supremum is taken over all disjoint nonempty subsets $A, B \subset \Omega_n$.
Hence, the authors define the dual Cheeger constant to be the following limit:
$ \overline{h}(G,m) = \lim_{n \rightarrow \infty} \overline{h}(\Omega_n). $
As the authors of \cite{B14} show, this limit exists and it is independent of the exhaustion.
Whilst Definition \ref{def:dualiso} does not involve any such limit, our dual Cheeger constant is equivalent to that in \cite{B14}; the difference lies in the underlying weighted graphs.

\subsection{Example of a non-bipartite graph with $\overline h(G, m)=1$}\label{ex:nonbip}
Consider the infinite graph $L$ shown on Figure \ref{fig:ladder1}; its vertexes are labelled by $v_0, v_1, \dots$ and $w_1, w_2,\dots$. The graph $L$ is not bipartite since it has a cycle of odd order. 
We set the weights as follows: $m(v_iw_i) =r^i$ and $m(v_iv_{i+1}) = \rho^i$, where $0<\rho<r <1$; besides, 
$m(v_0w_1)=1$. 
\begin{figure}[h]
\begin{center}
\includegraphics[scale=0.4]{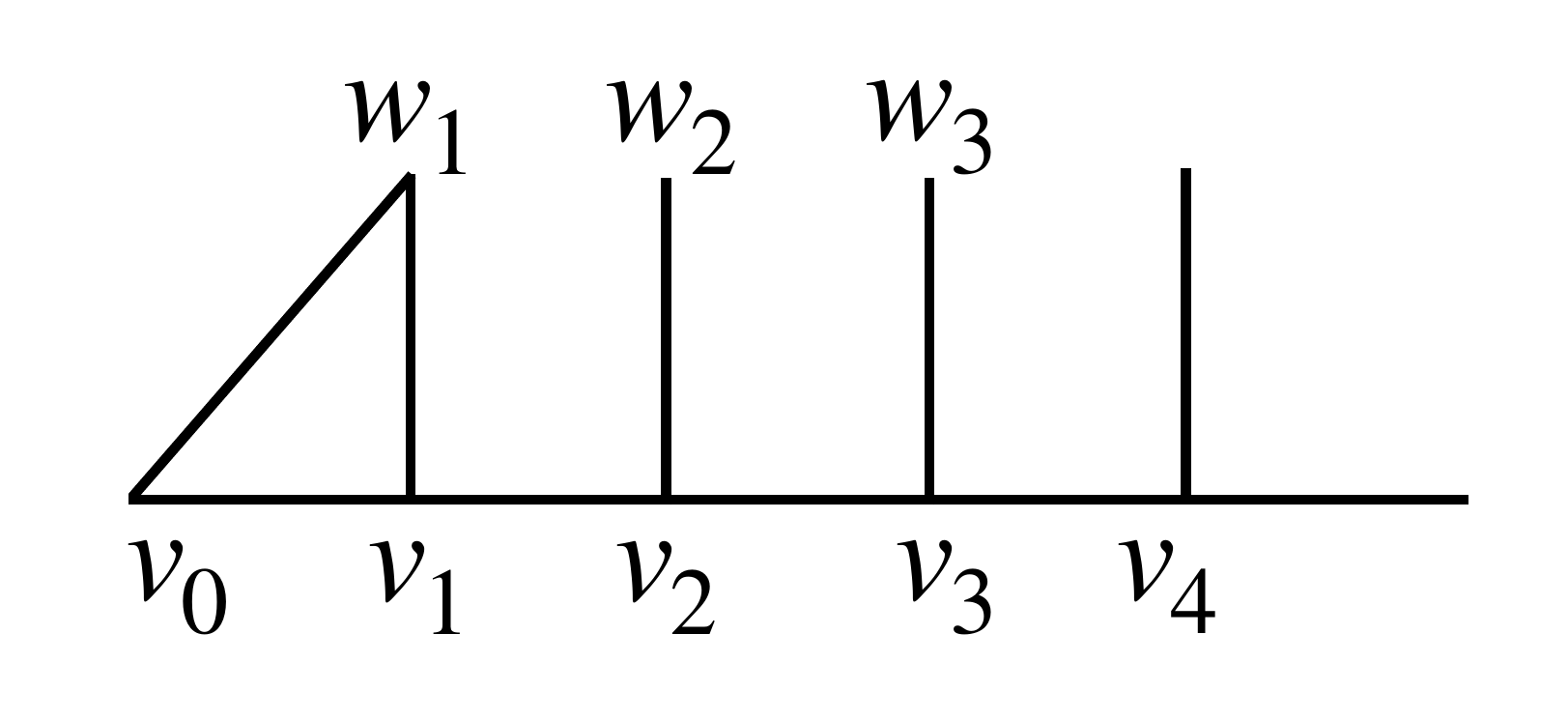}
\caption{Non-bipartite graph $L$.}\label{fig:ladder1}
\end{center}
\end{figure}
For $i> 1$ we have $m(v_i)=\rho^{i-1}+\rho^i+r^i$ and $m(w_i)=r^i$. Therefore, taking $A_i=\{v_i\}$ and $B_i=\{w_i\}$, where $i>1$, we have 
$$\overline h(A_i, B_i) = \frac{2r^i}{2r^i+\rho^{i-1}+\rho^i} \to 1.$$
Thus, we obtain $\overline h(L, m)=1$.

\subsection{The same graph $L$ with a different weight function} Consider now the following weight function $m_1$ on the graph $L$ of the previous example.  The function $m_1$ is defined similarly to $m$ with the only difference that now $\rho=r$, where $0<r<1$. In more detail, $m_1(v_iw_i)=r^i=m(v_iv_{i+1})$, and $m_1(v_0 w_1)=1$. Note that $m(w_i)=r^i$ for $i>1$, and $m(v_i)=r^{i-1}+2r^i$ for $i>0$. Besides, $m(w_1)=1+r$ and 
$m(v_0)=2$. 

Suppose that $A, B\subset \V(G)$ are disjoint sets of vertexes and for some $i\ge 1$ one has 
$v_i\in A$ and $w_i\in B$. Consider the following modifications of the pair $A, B$. 
Let $A_1=A\cup \{w_i\}$ and 
$B_1=B-\{w_i\}$. Besides, let $A_2= A$ and $B_2=B-\{w_i\}$.
Since the vertex $w_i$ is connected to the vertex $v_i$ only, by examining formula (\ref{dualch}) we easily see 
that $\overline h(A_1, B_1)< \overline h(A, B)$ and $\overline h(A_2, B_2)< \overline h(A, B)$. Thus, since we are interested in pairs $A, B$ giving maximum to the dual Cheeger constant (\ref{dualch}), we may always assume that for each vertex $v_i\in A$ the corresponding vertex $w_i$ belongs to $B$, and vice versa. 

Next, for a pair of disjoint subsets $A, B\subset \V(G)$, suppose that for some $i$ one has $v_i, v_{i+1}\in A$ and $w_i, w_{i+1}\in B$. We can modify the sets by swapping the points $v_{i+1}$ and $w_{i+1}$, i.e. 
$A'=A-\{v_{i+1}\}\cup \{w_{i+1}\}$ and $B'=B-\{w_{i+1}\}\cup \{v_{i+1}\}$. Then examining (\ref{dualch}) we see that $\overline h(A', B')> h(A, B)$. Thus, we may only consider pairs of disjoint subsets $A, B$ such that no neighbouring vertexes $v_i, v_{i+1}$ lie in the same subset $A$ or $B$. 

As another remark, consider a pair of disjoint subsets $A, B\subset \V(G)$ such that $v_i\in A$ and $w_i\in B$ and modify it as follows $A'=A\cup \{w_{i+1}\}$ and $B'= B\cup \{v_{i+1}\}$. Then we easily see from (\ref{dualch}) that $\overline h(A', B')>\overline h(A, B)$. 

Therefore, to determine $\overline h(G, m)$ we need to consider only pairs of subsets $A, B$ 
where $A$ contains all points $v_i$ with $i=k+2\ell$, where $\ell\ge 0$, and all points $w_i$ with 
$i=k+2\ell+1$, where $\ell\ge 0$; the set $B$ is defined similarly with the letters $v$ and $w$ interchanged. 
Here $k$ is a fixed integer $k\ge 1$. 
One finds that the dual Cheeger ratio 
$\overline h(A, B)=\frac{4r}{3r+1}$ 
does not depend on $k$.  

In the case $k=0$ we have to consider the slightly modified sets $A= \{v_0, v_1, w_2, v_3, \dots\}$ and $B=\{w_1, v_2, w_3, \dots\}$. Computing the dual Cheeger ratio (\ref{dualch}) gives $\overline h(A, B)=r$. 
Since $r< \frac{4r}{3r+1}$ we conclude that $\overline h(G, m)=\frac{4r}{3r+1}<1$. 

These two examples illustrate the possibility that the dual Cheeger constant can be maximal and equal 1 for one weight function and be smaller than 1 for another weight function.

\section{The Cheeger and dual Cheeger inequalities} \label{sec:ineqs} 

In this section, we we show that the Cheeger constant and the dual Cheeger constant control the spectral gap and the top of the spectrum of the Laplacian, respectively. 
In particular, we prove the Cheeger inequalties and the dual Cheeger inequalities for countable weighted graphs with summable weight function.
These inequalities give estimates on the spectral gap $\lambda_{gap}(\Delta)=\inf\{\lambda>0, \, \lambda\in \sigma(\Delta)\}$ and 
	the top of the spectrum $\lambda_{top}(\Delta)=\sup\{\lambda\in \sigma(\Delta)\}$.

\begin{theorem}[Cheeger and dual Cheeger inequalities]\label{thm:ci}
	For any weighted graph $(G, m)$ with summable weight function one has
	\begin{align}
	1 - \sqrt{ 1 - h(G,m)^2} \, \leq\,  &\lambda_{gap}(\Delta) \, \leq\,  2 \cdot h(G,m),\label{in:cheeger}\\
	2 \cdot \overline{h}(G,m) \, \leq\,  &\lambda_{top}(\Delta) \, \leq\,  1 + \sqrt{1 - \big(1 - \overline{h}(G,m)\big)^2 }.\label{in:dual}
	 \end{align}
	\end{theorem}

\begin{proof} For a subset $S\subset \V(G)$ satisfying $m(S) \leq m(S^\compl)$ define the function $f_S \in L^2(G,m)$
such that $f_S(v)= m(S)^{-1}$ for $v\in S$ and $f_S(v)=-m(S^\compl)^{-1}$ for $v\in S^c$. Since $f_S \perp {\bf 1}$, we may apply (\ref{eq:char}) to get 
\begin{align*}
			\lambda_{gap}(\Delta)
			&\leq \frac{\bra \Delta f_S,f_S \ket}{\bra f_S,f_S\ket} =
			m(\partial S) \cdot \big( \frac{1}{m(S)} + \frac{1}{m(S^\compl)} \big) 
			\leq 2\cdot h(S).
			\end{align*}
		Since this is true for all nonempty $S\subset \V(G)$ satisfying (\ref{eq:half}), we obtain the right inequality in (\ref{in:cheeger}).

To prove the left inequality in (\ref{in:dual}), for a pair $A, B\subset \V(G)$ of nonempty disjoint subsets define the function $f_{A,B} \in L^2(G,m)$ as follows: (a) $f_{A,B}(v)=1$ for $v\in A$, (b) $f_{A,B}(v)=-1$ for $v\in B$ and 
(c) $f_{A,B}(v)=0$ for $v\in \V(G)-(A\cup B)$. Using the characterisation of $\lambda_{top}(\Delta)$ in terms of Rayleigh quotients, we have
	\begin{align*}
		\lambda_{top}(\Delta)
		&\geq \rayleigh{f_{A,B}} 
		= 2 \cdot \overline{h}(A,B) + h(A\cup B) 
		 \geq 2 \cdot \overline{h}(A,B) .
		\end{align*}
	Since this is true for all nonempty, disjoint subsets $A,B \subset \V(G)$, the left inequality in (\ref{in:dual}) follows.
	
	To continue with the proof of Theorem \ref{thm:ci} we shall need to prepare certain tools. 
	The proof will be completed after Lemma \ref{lm16}.

\begin{lemma}[Co-area formulae] \label{lem:minm}
		For a function $f\colon \V(G)\to \real$ and for $t\in \real$ we shall denote
$P_t(f) = \{v\in \V(G); f(v)>t\}$. Then for every $f \in L^2(G,m)$ one has
		\begin{itemize}
			\item[(a)] $\displaystyle \;\; \int_0^\infty  m\big(P_t(f^2)\big) dt = \sum_{v\in \V(G)} m(v) \cdot f(v)^2, $ 
			\item[(b)] $\displaystyle \int_0^\infty m\big(\partial P_t(f^2)\big) dt = \sum_{uv \in \E(G)} m(uv) \cdot |f^2(u) - f^2(v)| .$
			\end{itemize}
		\end{lemma}

\begin{proof} The superlevel sets of $f^2$ satisfy
			$v \in P_t(f^2) \iff \mathbbm{1}_{(t,\infty)}\big(f^2(v)\big) = 1.$
			Therefore, 
			\begin{align*}
				&{}\int_0^\infty m\big(P_t(f^2)\big) dt
				= \int_0^\infty \sum_{v \in \V(G)} m(v) \cdot \mathbbm{1}_{(t,\infty)}\big(f^2(v)\big)dt \\
				&= \sum_{v\in \V(G)} m(v) \cdot \int_0^\infty \mathbbm{1}_{(t,\infty)}\big(f^2(v)\big) dt 
				=  \sum_{v\in \V(G)} m(v) \cdot  f^2(v).
				\end{align*}
			
			For any two vertices $u$ and $v$, define 
			$ I_{uv} = \Big[ \min\big\{ f^2(u) , f^2(v)\big\},\max\big\{ f^2(u), f^2(v) \big\} \Big)$. Then,
			\begin{align*}
				&\int_0^\infty m\big(\partial P_t(f^2)\big) dt 
				= \int_0^\infty \sum_{uv \in \E(G)} m(uv) \cdot \mathbbm{1}_{I_{uv}}(t) dt \\
				&= \sum_{uv \in \E(G)} m(uv) \cdot \int_0^\infty  \mathbbm{1}_{I_{uv}}(t) dt  
				= \sum_{uv \in \E(G)} m(uv) \cdot | g^2(u) -  g^2(v)|.
				\end{align*} \end{proof}
				
Consider the operator $Q=2I-\Delta\colon L^2(G, m)\to L^2(G, m)$. We have $\lambda\in \sigma(\Delta)$ if and only if $2-\lambda\in \sigma(Q)$. Therefore, $2- \lambda_{top}(\Delta)$ equals the bottom of the spectrum of $Q$, i.e. 
\begin{eqnarray}
2-\lambda_{top}(\Delta) = \inf \left\{\frac{\langle Qf, f\rangle}{\langle f, f\rangle}; f\not=0\right\}
\end{eqnarray}
A straightforward computation shows that for $f\in L^2(G, m)$, one has
\begin{eqnarray}\label{lem:zero1}
\langle Qf, f\rangle = \sum_{vw\in \E(G)}m(vw)\cdot (f(v)+f(w))^2.
\end{eqnarray}

\begin{lemma} \label{lem:two1} For any function $f\in L^2(G, m)$, $f\not\equiv 0$, one has the inequalities 
\begin{eqnarray}\label{quad}
1 - \sqrt{ 1 - h(f)^2} \leq \rayleigh{f} \leq 1 + \sqrt{ 1 - h(f)^2}.
\end{eqnarray}
	where the number $h(f)\ge 0$ is defined as the infimum of the
Cheeger ratios $h(S)=m(\partial S)m(S)^{-1}$ taken over all nonempty subsets $S\subset \{v; f(v)\ne 0\}$.
	\end{lemma}

\begin{proof}
Consider the expressions
$A = \bra Qf, f \ket  \cdot \bra\Delta f,f \ket $ and $B = \bra Qf, f \ket  \cdot \bra f,f \ket.$	
From (\ref{eq:delhelp}), (\ref{lem:zero1}) and the Cauchy-Schwarz inequality we obtain
	\begin{align*}
		 A^{1/2} 
		 &=   \Big( \sum_{vw \in \E(G)} m(vw) \cdot \big( f(v) + f(w) \big)^2\Big)^{1/2}  \cdot  \Big(\sum_{vw \in \E(G)} m(vw) \cdot \big( f(v) - f(w) \big)^2\Big)^{1/2} \\
		&\geq  \sum_{vw\in\E(G)} m(vw) \cdot |f(v)^2 - f(w)^2 | = \int_0^\infty m\big(\partial P_t(f^2) \big) dt\\
		&= \int_0^{t^*} \frac{m\big(\partial P_t(f^2)\big)}{m\big(P_t(f^2)\big)} \cdot m\big(P_t(f^2)\big )dt \\
		 &\geq h(f^2) \cdot \int_0^\infty m(P_t(f^2)) dt 
		= h(f) \cdot \bra f,f \ket.
		\end{align*}
		The two bottom lines use Lemma \ref{lem:minm}; the finite or infinite value $t^*$ is defined 
		by $ t^* = \sup \{ f^2\}; $
		it is introduced in order to avoid division by 0. 
		Thus we have $A\ge h(f)^2\cdot ||f||^4$. We also have 
		$B= \Big(2- \frac{\bra \Delta f, f\ket}{\bra f, f\ket}\Big)\cdot ||f||^4$. Dividing we obtain that 
		the quotient $\rayleigh{f}= AB^{-1}$ satisfies the inequality
		 $$AB^{-1} \ge \frac{h(f)^2}{2- AB^{-1}}.$$
		Solving this quadratic inequality for $AB^{-1}$ gives (\ref{quad}). 
			\end{proof}
		
Next, we observe that for any nonzero $g\in L^2(G, m)$ one can find a real number $\tau=\tau(g)$ such that 
\begin{eqnarray}\label{15}
m(g^{-1}(-\infty, \tau))\, \le \, m(\V(G))/2 \quad \mbox{and}\quad m(g^{-1}(\tau, \infty))\, \le \, m(\V(G))/2.
\end{eqnarray}
Indeed, it is easy to see that one can take $\tau= \sup\{t; \, m(g^{-1}(-\infty, t))< m(\V(G))/2\}$. 


Define the functions $g_+, g_-\in L^2(G, m)$ by
$$g_+(v)= \max\{g(v)-\tau(g), 0\}\quad \mbox{and}\quad g_-(v)= \max\{\tau(g) - g(v), 0\}.$$
Then 
\begin{eqnarray}\label{16}
g=g_+-g_- +\tau, \quad  \quad g_+ g_- =0,\end{eqnarray}
 and $ h(G, m)\le \min\{h(g_+), h(g_-)\}$ because of (\ref{15}).  
 If $g\perp \mathbf 1$ then $\bra g, g\ket $ can be estimates as follows:
 \begin{eqnarray}\label{17}
 \bra g, g\ket &=&\sum m(v)g(v)^2\, \le \, \sum m(v) (g(v)-\tau)^2= \sum m(v)(g_+(v)^2+g_-(v)^2)\\
 &= &
 \bra g_+, g_+\ket + \bra g_-, g_-\ket.\nonumber \end{eqnarray}
 Besides, 
 \begin{eqnarray}\label{18}
 \bra\Delta g, g\ket &=& \sum m(vw) (g(u)-g(v))^2 \nonumber\\
 &\ge& \sum m(uv)((g_+(u)-g_+(v))^2 + (g_-(u)-g_-(v))^2) =  \bra\Delta g_+, g_+\ket + \bra\Delta g_-, g_-\ket
 \end{eqnarray}
 Indeed, if the vertexes $u, v$ are such that $g(v)<\tau<g(u)$ then 
 \begin{eqnarray*}(g(u)-g(v))^2 &=& (g_+(u)-g_-(v))^2 \, \ge \, 
 g_+(u)^2 + g_-(v)^2 \\
 &=& (g_+(u)-g_+(v))^2 +(g_-(u)-g_-(v))^2.\end{eqnarray*}
 Thus, for $g\perp \mathbf 1$, using (\ref{17}) and (\ref{18}) we obtain
 \begin{eqnarray*}
 \frac{\bra \Delta g, g\ket}{\bra g, g \ket} &\ge& \min\{ \frac{\bra \Delta g_+, g_+\ket}{\bra g_+, g_+ \ket}, 
 \frac{\bra \Delta g_-, g_-\ket}{\bra g_-, g_-\ket}\}\\
 &\ge & \min\{ 1- \sqrt{1-h(g_+)^2}, \, 1- \sqrt{1-h(g_-)^2}\}\, \ge\,  1-\sqrt{1-h(G, m)^2}.
 \end{eqnarray*}
 This proves the left inequality in (\ref{in:cheeger}). 
 
 Finally, we prove the right inequality (\ref{in:dual}), i.e. the upper bound for the $\lambda_{top}$.
 We shall use the idea of  \cite{B14} and adopt their arguments to our situation. 
 
 Given a nonzero function $f\in L^2(G, m)$ we consider the auxiliary weighted graph $(G_f, m_f)$ 
 which is constructed as follows. For each vertex $v\in \V(G)$ with $f(v)\not=0$ we create an additional vertex $v'$ and the vertex set of the new graph $G_f$ equals $\V(G_f)= \V(G)\cup \{v'; \, v\in \V(G), f(v)\not =0\}.$ 
 
 Next, we describe the set of edges of $G_f$. We remove every edge $vw\in \E(G)$ with $f(v)f(w)>0$ and replace it with the edges $vw'$ and $v'w$ of $\E(G_f)$. All edges $vw$ of $G$ satisfying $f(v)f(w)\le 0$ are included into $E(G_f)$. 
 
 The weight function $m_f$ on $G_f$ is defined as follows: firstly, $m_f(v'w)=m_f(vw')=m(vw)$ and, secondly, for every edge $vw\in \E(G)$ with $f(v)f(w)\le 0$ we set $m_f(vw)=m(vw).$

Note that the weights of vertexes $v\in \V(G)$ remain unchanged: $m_f(v)=m(v)$. Besides, the weights of the new vertexes $v'$ satisfy $m_f(v')\le m(v)$. 

Consider the function $f'\in L^2(G_f, m_f)$ defined by $f'(v')=0$ and 
$f'(v)=|f(v)|$ for $v\in \V(G)$.
 
\begin{lemma}\label{lm16} One has
\begin{eqnarray}\label{21}
\frac{\bra Q f, f \ket}{\bra f, f\ket} \ge \frac{\bra \Delta f', f'\ket_f}{\bra f', f'\ket_f},
\end{eqnarray}
where $\bra \cdot, \cdot\ket_f$ denotes the scalar product in $L^2(G_f, m_f)$. 
\end{lemma}
\begin{proof}
Firstly, 
$\bra f, f\ket=\sum_{v\in \V(G)} m(v)f(v)^2 = 
\sum_{v\in \V(G_f)} m_f(v)f'(v)^2=\bra f', f'\ket_f.$ 
Next we show 
\begin{eqnarray}\label{in:rel}
\bra Q f, f\ket \ge \bra \Delta f', f'\ket_f.
\end{eqnarray}
If $vw$ is an edge in $G$ with $f(v)f(w)>0$, then 
\begin{eqnarray*}
m(vw)(f(v)+f(w))^2 &\ge& m(vw)(f(v)^2+f(w)^2) = m_f(vw')f'(v)^2 + m_f(v'w) f'(w)^2\\
&=& m_f(vw')(f'(v)-f'(w'))^2 +m_f(v'w)(f'(v')-f'(w))^2.
\end{eqnarray*}
Besides, for an edge $vw\in \E(G)$ with $f(v)f(w)\le 0$ one has 
$$m(vw)(f(v)+f(w))^2 = m_f(vw)(f'(v)-f'(w))^2.$$
Incorporating the above information into (\ref{eq:delhelp}) and (\ref{lem:zero1}) we obtain (\ref{in:rel}) and hence 
(\ref{21}). 
\end{proof}

We intend to use the left inequality in (\ref{quad}) applied to $f'$ viewed as an element of $L^2(G_f, m_f)$. 
 The number $h_f(f')$ is defined as the infimum of the ratios $h_f(S)=m_f(\partial S)m_f(S)^{-1}$, where $S$ runs over subsets of the support of $f'$. In our case the support of $f'$ lies in $\V(G)\subset \V(G_f)$ and can be represented as the disjoint union $\supp(f)_+\sqcup \supp(f)_-$, where $\supp(f)_+$ is the set of all vertexes $v\in \V(G)$ where $f$ is positive and $\supp(f)_-$ is the set of all vertexes $v\in \V(G)$ with $f(v)<0$. Thus, any subset $S\subset \supp (f)$ is the disjoint union $S=S_+\sqcup S_-$ where $S_\pm = S\cap \supp(f)_\pm$. In the graph $G_f$ there are no edges internal to $S_+$ and there are no edges internal to $S_-$. Thus, we obtain
 $$m_f(S) = m(S) = m_f(\partial_f S) + 2 m(S_+, S_-).$$
 Therefore, we see that 
 $$h_f(S) =\frac{m_f(S)- 2m(S_+, S_-)}{m_f(S)}= 1- \frac{2m(S_+, S_-)}{m(S_+)+m(S_-)}=1 -\h(S_+, S_-).$$
 Taking the infimum over $S\subset \supp f$ we obtain the inequality 
 \begin{eqnarray}\label{nearly}
 h_f(f') \ge 1- \h(G, m). 
 \end{eqnarray}
 
 Now we can obtain the desired upper bound for $\lambda_{top}=\lambda_{top}(\Delta)$. We have 
 \begin{eqnarray*}
 2-\lambda_{top} = \inf_{f\not=0} \frac{\bra Qf, f\ket}{\bra f, f\ket}\ge \inf_{f\not=0}\frac{\bra \Delta f', f'\ket_f}{\bra f', f'\ket_f} \ge 1- \sqrt{1- h_f(f')^2}
 \ge 1- \sqrt{1- (1-\h(G, m))^2}.
 \end{eqnarray*}
 This completes the proof of Theorem \ref{thm:ci}. 
\end{proof}
\begin{remark} Theorem 13.4 from the book \cite{Keller} is a corollary of the left part of the first inequality of Theorem \ref{thm:ci}. 

\end{remark} 

\section{A new combinatorial invariant}\label{kk}

In this section we introduce a new combinatorial invariant of countable graphs with summable weight functions. We show that this invariant is a measure of spectral asymmetry. We describe relations between the new invariant and the Cheeger and the dual Cheeger constants. 

\subsection{Definition} Let $G$ be a connected countable graph and let $m: \E(G)\to (0, \infty)$ be a summable weight function. 
For any vertex $v\in \V(G)$ and any subset $S\subset \V(G)$ we write
\begin{eqnarray}\label{pa}
m_S(v) = \sum_{w\in S}m(vw) \quad \mbox{and}\quad p_S(v) = \frac{m_S(v)}{m(v)}.
\end{eqnarray}
If we think of the underlying weighted graph as a Markov chain, then $p_S(v)$ is the probability that the particle starting at $v$ 
ends up in $S$ in one step. If $A\sqcup B =\V(G)$ is a partition of the vertex set with $A\not=\emptyset\not= B$, then 
$$p_A(v) +p_B(v) =1\quad \mbox{for any}\quad v\in \V(G).$$
\begin{definition}\label{def:kab}
For a partition $A\sqcup B =\V(G)$ of the vertex set we define 
\begin{eqnarray} \label{kab}
\k(A, B) = \max\{\sup_{v\in A} p_A(v), \sup_{w\in B}p_B(w)\}.\end{eqnarray}
Finally, we associate the following constant $\k(G, m)$ to the weighted graph:
\begin{eqnarray}
\k(G, m) = \inf \k(A, B);
\end{eqnarray}
the infimum is taken over all partitions $A\sqcup B =\V(G)$. 
\end{definition}

The inequality $\k(G, m)>c$ means that for any partition $A\sqcup B=\V(G)$ there exists either $v\in A$ or $w\in B$ such that one of the inequalities $p_A(v)>c$ or $p_B(w)>c$ holds. 
The constant $\k(G, m)$ is the supremum of the numbers $c$ for which this graph property
holds.

\subsection{Characterisation of bipartite graphs} 
\begin{theorem} \label{thm:bip} One has $\k(G, m)\ge 0$ for any weighted graph $(G, m)$. Moreover, $\k(G, m)=0$ if and only if the graph $G$ is bipartite. 
\end{theorem}
\begin{proof}
The inequality $\k(G, m)\ge 0$ is obvious from the definition. 

If $G$ is bipartite and $\V(G)=A\sqcup B$ is a partition such that every edge goes from $A$ to $B$, then
$p_A(v)=0$ for $v\in A$ and $p_B(w)=0$ for $w\in B$ implying $\k(A, B)=0$ and therefore $\k(G, m)=0$. 

Let $(G, m)$ be a summable weighted graph with $k(G, m)=0$. Let $A_n\sqcup B_n =\V(G)$ be a sequence of partitions with 
$\k(A_n, B_n)\to 0$. 
Suppose that it is not bipartite. Then $G$ admits a cycle $C$ of odd length. For each edge $vw$ of the cycle $C$ consider the quantity
\begin{eqnarray}
a(vw) = \min\left\{ \frac{m(vw)}{m(v)}, \frac{m(vw)}{m(w)}\right\} > 0
\end{eqnarray}
and choose $n$ so large that 
\begin{eqnarray}\label{in:g}
a(vw) > \k(A_n, B_n) \quad \mbox{for any edge $vw$ of $C$.}
\end{eqnarray}
Using the definition of $\k(A_n, B_n)$ we see that every edge $vw$ of $C$ connects a vertex of $A_n$ with a vertex of $B_n$.
Indeed, if $v, w\in A_n$ then $\k(A_n, B_n) \ge p_{A_n}(v) \ge \frac{m(vw)}{m(v)}=a(vw)$ contradicting (\ref{in:g}). 
This leads to contradiction with the fact that the cycle $C$ has odd length. 
\end{proof}

\begin{remark} Recall that the equality $\overline h(G, m)=1$ only partially characterises the class of bipartite graphs, see (\ref{bip}) and Example 
\ref{ex:nonbip}. In contrast, by Theorem \ref{thm:bip}, the equality $\k(G, m)=0$ gives a complete characterisation. 
\end{remark}

\begin{example}
Let $G$ be a complete graph on $n+1$ vertexes. We equip $G$ with the counting weight function, i.e. $m(vw)=1$ for every edge. 
If $A\sqcup B=\V(G)$ is a partition with $|A|\le |B|$ then 
$$\k(A, B) = \max \left\{ 
\frac{|A|-1}{n}, \frac{|B|-1}{n}
\right\} = \frac{|B|-1}{n}
$$
and we obtain 
\begin{eqnarray*}
\k(G, m) = n^{-1} \cdot \lceil \frac{n-1}{2}\rceil.
\end{eqnarray*}
Thus, $\k(G, m)= 1/2 - 1/(2n)$ for $n$ odd and $\k(G, n) = 1/2$ for $n$ even. 
\end{example}

\begin{example}
Let $C_n$ be the cycle of order $n$ with the counting weight function $m$, i.e. the weight of each edge equals 1. 
Then $\k(C_n, m)=0$ for $n$ even (since then $C_n$ is bipartite) and $\k(C_n) = 1/2$ for $n$ odd. Indeed, 
for $n$ odd, for any partition $A\sqcup B$ of the vertex set one of the sets $A$ or $B$ must contain two adjacent vertexes and therefore one of the numbers $p_A(v)$ or $p_B(w)$ is $\ge 1/2.$
\end{example}

It is obvious that $\k(G, m)\le 1$. In all examples known to us we have $\k(G, m)\le 1/2$, but we do not know if this inequality is true in general. 

\subsection{Relation with the dual Cheeger constant} Next we describe inequalities relating the new invariant $\k(G, m)$ with $\overline h(G, m)$. We shall frequently use the inequlity 
\begin{eqnarray}\label{frac}
\min \{\frac{a}{c}, \frac{b}{d}\}\, \le \, \frac{a+b}{c+d}\, \le \, \max \{\frac{a}{c}, \frac{b}{d}\},
\end{eqnarray}
where $a, b, c, d>0$. 

For a subset of vertexes $A\subset \V(G)$ we shall denote 
$$R_A= \frac{2\cdot m(A, A)}{m(A)}.$$

\begin{lemma}
For any partition $A\sqcup B=\V(G)$ of the set of vertexes of a summable weighted graph $(G, m)$ one has
\begin{eqnarray}\label{frac1}
\min\{R_A, R_B\}\le 1- \overline h(A, B) \le \max\{R_A, R_B\}\le \k(A, B).
\end{eqnarray}
\end{lemma}
\begin{proof}
Since $m(A)= 2m(A, A)+m(A, B)$ we have 
$$1-\overline h(A, B) = 2\cdot \frac{m(A, A)+m(B, B)}{m(A)+m(B)}$$ and using (\ref{frac}) we obtain the left and the central inequalities in (\ref{frac1}). The right inequality in (\ref{frac1}) follows from 
$R_A \le \sup_{v\in A} \{p_A(v)\}$ and $R_B \le \sup_{w\in B} \{p_B(w)\},$ which are consequences of (\ref{frac}) as well.  
\end{proof}

Next we state a relationship between $\k(G, m)$ and the dual Cheeger constant. 
\begin{corollary} For any summable weighted graph $(G, m)$ one has 
\begin{eqnarray}
\overline h(G, m)+\k(G, m) \ge 1. 
\end{eqnarray}
\end{corollary}
\begin{proof}
The statement follows by taking infimum of the both sides of (\ref{frac1}). 
\end{proof}

Below is a relation between the quantities $R_A$ and $R_B$ and the Cheeger constant: 

\begin{lemma} The Cheeger constant $h(G, m)$ of a summable weighted graph $(G, m)$ equals 
\begin{eqnarray}
1- \sup_{A, B} \min\{R_A, R_B\},
\end{eqnarray}
where the supremum is taken with respect to all nonempty $A, B\subset \V(G)$ that partition $\V(G)$. 
\end{lemma}
\begin{proof}
We may write 
$$R_A = \frac{2\cdot m(A, A)}{m(A)} = \frac{m(A)-m(\partial A)}{m(A)} = 1- h(A)$$
and similarly $R_B = 1-h(B)$. Thus, 
$$
h(G, m) = \inf_{A, B} \max \{h(A), h(B)\} = \inf_{A, B} \max \{1- R_A, 1- R_B\} = 
1- \sup_{A, B} \min\{R_A, R_B\}.
$$
\end{proof}

\section{Measure of asymmetry of the spectrum}
	
Recall that for a pair of non-empty subsets $X$ and $Y$ of a metric space,  
their Hausdorff distance 
$d_H(X,Y)$ is defined as 
$${\displaystyle d_{\mathrm {H} }(X,Y)=\max \left\{\,\sup _{x\in X}d(x,Y),\,\sup _{y\in Y}d(X,y)\,\right\},\!}$$
Equivalently,
$$d_{\mathrm H}(X,Y)=\inf\{
\epsilon\ge 0; X\subset Y_\epsilon\,\,  \mbox{and}\, \, Y\subset X_\epsilon\}$$ 
where
${\displaystyle X_{\epsilon }=\bigcup _{x\in X}\{z\in M\,;\ d(z,x)\leq \epsilon \}}.$

In this section we consider a summable weighted graph $(G, m)$ and 
the spectrum $\sigma(\Delta)\subset [0,2]$ of its Laplacian $\Delta: L^2(G,m)\to L^2(G, m)$, see 
(\ref{def:lap}) and Lemma \ref{lem:prop}. 
The main result of this section, Theorem \ref{thm:mains}, describes the asymmetry of the spectrum $\sigma(\Delta)$ in terms of the invariant 
$\k(G, m)$ introduced in the previous section.

\begin{theorem} \label{thm:mains}
	Let $\mathcal R: [0,2]\to [0, 2]$ denote the reflection $\mathcal R(x)=2-x$. 
	The Hausdorff distance between the spectrum of the Laplacian $\sigma(\Delta)$ and its reflection 
	$\mathcal R(\sigma(\Delta))$ is at most $2 \cdot \k(G,m)$, i.e. 
	\begin{eqnarray}\label{Hausdorff}
	d_{\mathrm {H} }(\sigma(\Delta),\mathcal R(\sigma(\Delta)))\leq 2 \cdot \k(G,m) . 
	\end{eqnarray}
	\end{theorem}

The following simple observation will be useful.

\begin{lemma}\label{lm:isometry}
For any subset $X\subset [0,2]$, the Hausdorff distance $d_{\mathrm {H} }(X,\mathcal R(X))$ equals 
$\inf\{\epsilon >0; \mathcal R(X) \subset X_\epsilon\}$.
\end{lemma}
\begin{proof} Since $\mathcal R$ is an involution and an isometry, the relation $\mathcal R(X)\subset X_\epsilon$ implies 
$$X=\mathcal R(\mathcal R(X))\subset \mathcal R(X_\epsilon)=\mathcal R(X)_\epsilon.$$
The statement now follows from the definition of the Hausdorff distance. 
\end{proof}

The proof of Theorem \ref{thm:mains} is completed by the end of this section. 
%
%
%

Recall that the Laplacian $\Delta\colon L^2(G, m) \to L^2(G, m)$ equals $I-P$ where $I$ is the identity operator and 
$P\colon L^2(G, m)\to L^2(G, m)$ is given by formula (\ref{pp}). 

Let $\psi\colon  \V(G) \rightarrow \real $ be a function satisfying $\psi^2\equiv 1$. Consider the associated partition
of the vertex set $A\sqcup B=\V(G)$ where $A=\psi^{-1}(1)$ and  $B=\psi^{-1}(-1)$. We obtain the decomposition of 
$L^2(G, m)$ into 
the direct sum of two Hilbert spaces
\begin{eqnarray}
L^2(G, m) \, =\, L^2(A, m)\oplus L^2(B, m).
\end{eqnarray}
Here $L^2(A, m)\subset L^2(G,m)$ is the space of functions 
$f\in L^2(G, m)$ with $\supp (f)\subset A$ and similarly for $L^2(B, m)$. Denote by 
$$\pi_A, \pi_B: L^2(G, m)\to L^2(G, m)$$
the projections
$$\pi_A(f)|_A = f|_A, \quad \pi_A(f)|_B =0,$$
and similarly for $\pi_B$. Clearly, the operators $\pi_A, \pi_B$ are self-adjoint. Write $$T_\psi=\pi_A - \pi_B \quad \mbox{and}\quad P_\psi = P_A + P_B,$$ viewed as operators 
$L^2(G, m) \to L^2(G, m)$, 
where 
$$P_A= \pi_A P \pi_A: L^2(A, m)\to L^2(A, m)\quad \mbox{and}\quad  P_B= \pi_B P \pi_B: L^2(B, m)\to L^2(B, m).$$ 
The operators $P_A$ and $P_B$ are self-adjoint since they are compositions of self-adjoint operators. 

\begin{lemma} \label{lem:conj}
	One has
	$$T_\psi^{-1} \Delta T_\psi = 2 \cdot I - \Delta - 2\cdot P_\psi$$
and therefore 
	\begin{eqnarray}\label{sigma}
	 \sigma(\Delta) = \sigma(2 \cdot I - \Delta - 2\cdot P_\psi). \end{eqnarray}
	\end{lemma}
	\begin{proof} Since $T_\psi^{-1} = T_\psi$ one has
	\begin{eqnarray*}
T_\psi^{-1} \Delta T_\psi &=& (\pi_A-\pi_B)(I-P)(\pi_A - \pi_B)\\ &=& I - (\pi_A-\pi_B)P(\pi_A-\pi_B)\\
&=& I +\pi_A P\pi_B +\pi_B P\pi_A - P_A - P_B\\
&=& I + P - 2 \cdot (P_A+P_B) \\ 
&=& 2\cdot I - \Delta -2 \cdot P_\psi.
	\end{eqnarray*}
	This proves the first claim of the Lemma.
	The second claim follows as well using the fact that the spectrum is invariant under conjugation.
	\end{proof}
	
	Next we show that the norm of the operator $P_\psi$ can be estimated using the combinatorial invariant 
	$\k(A, B)$, which is introduced in \S \ref{kk}.

\begin{lemma} \label{thm:phnorm} Let $(G, m)$ be a summable weighted graph. Let $\psi \colon \V(G) \rightarrow \real$ be a function satisfying $\psi^2\equiv 1$ and let  
$A\sqcup B=\V(G)$ be the associated partition of the vertex set $\V(G)$, where $\psi |_A \equiv 1$ and $\psi |_B \equiv -1$.
 Then one has
 \begin{eqnarray}\label{35}
   || P_\psi || \le   \k(A,B). \end{eqnarray} 	
\end{lemma}
\begin{proof} Since $P_\psi$ is self-adjoint, its norm $||P_\psi ||$ equals the supremum of 
$$\frac{|\bra P_\psi f, f\ket|}{\bra f, f\ket}$$
taken over all nonzero $f\in L^2(G, m)$, see \cite{Kr}, \S 9.2. By construction, $P_\psi$ is the direct sum of the operators $P_A$ and $P_B$ and we show below that 
\begin{eqnarray}\label{36}
||P_A|| \le \sup_{v\in A} \{ p_A(v)\}\quad \mbox{and} \quad ||P_B|| \le \sup_{v\in B} \{ p_B(v)\},
\end{eqnarray}
where for $v\in A$ the quantity $p_A(v)$ is defined as $m(v)^{-1} \sum_{w\in A}m(vw)$ and similarly for $p_B(v)$. Clearly, due to the definition of $\k(A,B)$, (\ref{36}) implies (\ref{35}) and thus we only need to prove the inequality (\ref{36}). 

For $f\in L^2(A, m)$ we have 
$(P_A f)(v) = m(v)^{-1} \cdot \sum_{w\in A} m(vw)f(w)$ and 
therefore 
\begin{eqnarray*}
\bra P_A f, f\ket &=& \sum_{v, w\in A} m(vw)f(v)f(w) \\
&\le& 1/2 \cdot \sum_{v, w\in A} m(vw)[f(v)^2 +f(w)^2]
= \sum_{v, w\in A} m(vw) f(v)^2 \\
&=& \sum_{v\in A} m(v) \cdot \left[ \sum_{w\in A} \frac{m(vw)}{m(v)}    \right]f(v)^2
= \sum_{v\in A} m(v) \cdot p_A(v) \cdot f(v)^2\\
&\le & \sup_{v\in A} \, \{ p_A(v)\}\cdot ||f||^2. 
\end{eqnarray*}
This gives the left inequality (\ref{36}); the right one follows similarly. 
\end{proof}

Lemma \ref{thm:phnorm} implies: 
\begin{corollary}\label{inf}
One has  $$\inf_\psi \big\{ ||P_\psi||  \big\}\leq \k(G,m), $$
	where the infimum is taken over all functions $\psi\colon \V(G) \rightarrow \real$ satisfying
	$\psi^2\equiv 1$. 

\end{corollary}

\begin{proof}[Proof of Theorem \ref{thm:mains}]   Using (\ref{sigma}) together with an obvious equality 
$\sigma(2I-\Delta))=\sigma(\mathcal R(\Delta)) = \mathcal R(\sigma(\Delta))$ 
and Theorem 4.10 from \cite{Kato}  we have
\begin{eqnarray*}
d_{\mathrm H}(\sigma(\Delta), \mathcal R(\sigma(\Delta))) 
= d_{\mathrm H}(\sigma(2I - \Delta - 2\cdot P_\psi), \sigma(2I-\Delta))\le 2\cdot ||P_\psi||.
\end{eqnarray*}
The LHS of this inequality is independent of $\psi$. Taking infimum with respect to $\psi$ and 
applying  Corollary \ref{inf} we arrive at (\ref{Hausdorff}). 
\end{proof}

\section{The spectrum of the infinite complete graph}\label{sec8}

We start with the following general remark about summable weighted graphs. 

\begin{lemma} Let $(G, m)$ be a summable weight graph. Consider the random walk operator
$P\colon L^2(G, m)\to L^2(G, m)$, where $$(Pf)(v) = \sum_w \frac{m(vw)}{m(v)}\cdot f(w).$$ 
Then $P$ is a Hilbert-Schmidt operator if and only if 
\begin{eqnarray}\label{39}
\sum_{v, w} \frac{m(vw)^2}{m(v)m(w)}< \infty. 
\end{eqnarray}

\end{lemma}
\begin{proof}
Let $g_v\in L^2(G, m)$ be the function taking the value $m(v)^{-1/2}$ at point $v$ and vanishing at all other points. 
Clearly, the system 
$$\{g_v; v\in \V(G)\}$$ is an orthonormal basis. The matrix coefficients of the operator $P$ in this basis are 
$$\bra Pg_w, g_v\ket = \sum_{a\in \V(G)}  m(a) (Pg_w)(a) g_v(a) = \frac{m(vw)}{\sqrt{m(v)m(w)}}.$$
Thus, condition (\ref{39}) is equivalent to the well-known criterion for Hilbert-Schmidt operators. 
\end{proof}

\begin{example}
Consider again the weighted graph $(G,m)$ of Example \ref{ex:comp1}, i.e. $G$ is the infinite complete graph with vertex set $\Bbb N$ and with the weights 
$$m(ij)=p_i p_j\quad \mbox{where }\quad 
p_1\ge p_2\ge p_3\ge \dots , \quad \sum_{i\ge 1} p_i=1.$$ 
The weights of vertexes are $m(i)=p_i q_i$, where $q_i=1-p_i$. 
Condition (\ref{39}) requires in this case convergence of the series
$$\sum_{i\not= j} \frac{(p_ip_j)^2}{p_iq_i p_jq_j} \le \sum_{i, j} \frac{p_ip_j}{q_iq_j} = \left(\sum_{i}p_iq_i^{-1}\right)^2 < q_1^{-2}.$$
We see that in this example the random walk operator $P$ is Hilbert-Schmidt and hence compact. 
\end{example}

\begin{corollary}
 The spectrum of the Laplacian $\Delta = I-P: L^2(G, m)\to L^2(G, m)$ of the infinite complete graph $(G, m)$ 
consists of an infinite sequence of eigenvalues converging to $1\in [0,2]$ and the point $1$ is the only point of the spectrum which is not an eigenvalue. 
\end{corollary}

Below we analyse further the infinite complete graph and give a more detailed information about its spectrum. 

The random walk operator $P$ is given by $$(Pf)(i) = \sum_{j\not=i}\frac{p_j}{q_j} f(j)=\sum_{j}\frac{p_j}{q_j} f(j)- 
\frac{p_i}{q_i}f(i).$$
Consider the eigenvalue equation $Pf_\lambda=\lambda f_\lambda$, i.e. 
\begin{eqnarray}\label{equa}
\sum_j \frac{p_j}{q_j} f_\lambda(j) = (p_iq_i^{-1} +\lambda)f_\lambda(i)
\end{eqnarray}
for any $i\in \Bbb N$. Therefore, for any $i\ge 1$ one has
$$(p_iq_i^{-1}+\lambda)f_\lambda(i) = (p_1q_1^{-1}+\lambda)f_\lambda(1). $$
Without loss of generality we may assume that $(p_1q_1^{-1}+\lambda)f_\lambda(1)=1$. 
We obtain an infinite increasing sequence of numbers  
$$\alpha_1\le \alpha_2 \le \dots < 0, \quad \alpha_i\to 0, \quad \mbox{where}\quad \alpha_i = -p_iq_i^{-1},$$
and the eigenfunction $f_\lambda$ satisfies 
$$f_\lambda(i) = (\lambda- \alpha_i)^{-1}, \quad i\ge 1.$$
The equation (\ref{equa}) becomes 
\begin{eqnarray}\label{eigen1}
\sum_{j=1}^\infty \frac{\alpha_j}{\alpha_j - \lambda} = 1 
\end{eqnarray}
and the condition $f_\lambda\in L^2(G, m)$ can be expressed as 
$\sum_{i=1}^{\infty}p_iq_i(\alpha_j - \lambda)^{-2} \ < \ \infty.$
Since $q_1\le q_i<1$, the latter condition is equivalent to  
\begin{eqnarray}\label{eigen2}
\sum_{j=1}^{\infty}p_j(\lambda -\alpha_j)^{-2} \ < \ \infty.
\end{eqnarray}
Corollary \ref{cor:sum} summarises the above arguments.
\begin{corollary}\label{cor:sum} A number $\lambda\in [0,2]$ is an eigenvalue of the random walk operator $P$ for the infinite complete graph if and only if the equation (\ref{eigen1}) and inequality (\ref{eigen2}) are satisfied. 
\end{corollary}
We shall see later that the condition (\ref{eigen2}) is automatically satistfied. 

As an example, consider $\lambda=1$. One has $\frac{\alpha_j}{\alpha_j - 1} =  p_j$ and hence equation (\ref{eigen1}) is satisfied. Besides, $\frac{p_j}{(1-\alpha_j)^2} = p_jq_j^2$ and (\ref{eigen2}) follows from $\sum_j p_j=1$ and $q_j<1$.  Thus, 
$\lambda=1$ is an eigenvalue. 

The next lemma describes the whole spectrum of $P$. 

\begin{lemma}\label{lm:spectrump}
Assume that all numbers $p_i$ are pairwise distinct, i.e. $p_1>p_2> p_3> \dots$. 
The random walk operator $P\colon L^2(G, m)\to L^2(G, m)$ has a unique positive eigenvalue $1$ and negative eigenvalues 
$\lambda_1<\lambda_2< \dots <0$
satisfying
$$\alpha_i <\lambda_i< \alpha_{i+1}, \quad i=1, 2, \dots. $$
Additionally, $0$ belongs to the spectrum (as an accumulation point) and is the only point of the spectrum which is not an eigenvalue. 
\end{lemma}
\begin{proof} Consider the function 
\begin{eqnarray}\label{function}
F(\lambda) = \sum_{j=1}^\infty \frac{\alpha_j}{\alpha_j -\lambda}, \quad \quad \lambda\in \Omega = \real -\{0, \alpha_1, \alpha_2, \dots\}.\end{eqnarray}
First, we verify that the series (\ref{function}) converges for all  $\lambda\in \Omega$. 
Indeed, if $\lambda>0$ then $\lambda - \alpha_j  > \lambda$ for all $j\ge 1$ and 
$$\sum_{j\ge 1} \frac{\alpha_j}{\alpha_j - \lambda } \, < \, \lambda^{-1} \sum_{j\ge 1} (-\alpha_j ) \, < \, \lambda^{-1}.$$
Consider now the case when $\lambda\in (\alpha_i, \alpha_{i+1})$. 
Then for $j\ge i+2$ one has $\alpha_j - \lambda \, >\, \alpha_{i+2}-\alpha_{i+1}$ and therefore 
$$\sum_{j\ge i+2}\frac{\alpha_j}{\alpha_j -\lambda}\, \le \, (\alpha_{i+2}-\alpha_{i+1})^{-1}\cdot \sum_{j\ge i+2}(-\alpha_j)\ < \ \infty.$$

For $\epsilon >0$, let $\Omega_\epsilon$ denote $\real - \cup_{j\ge 1} B(\alpha_j, \epsilon)$, where $B(\alpha_j, \epsilon)$ stands for the open ball with centre $\alpha_j$ and radius $\epsilon$. 
The arguments of the preceding paragraph applied to the series of derivatives
\begin{equation}\label{der}
\sum_{j=1}^\infty \frac{\alpha_j}{(\alpha_j -\lambda)^2}
\end{equation}
show that the series (\ref{der}) converges uniformly on $\Omega_\epsilon$. 

Therefore, we obtain that the function 
$F(\lambda)$ is differentiable and its derivative $F'(\lambda)$ is given by the series (\ref{der}) for all $\lambda\in \Omega$. 

Each term of the series (\ref{der}) is negative; this implies that $F(\lambda)$ is monotone decreasing on every interval contained in $\Omega$. Thus, equation (\ref{eigen1}) may have at most one solution in any such interval. 

Consider the behaviour of $F(\lambda)$ on one of the intervals $\lambda \in (\alpha_i, \alpha_{i+1})$. It is obvious that for 
$\lambda\to \alpha_i$ the function $F(\lambda)$ tends to $+\infty$ and for $\lambda\to \alpha_{i+1}$ one has $F(\lambda)\to -\infty$. 

There are also two infinite maximal intervals of continuity $(-\infty, \alpha_1)$ and $(0, \infty)$. The function $F(\lambda)$ is negative for $\lambda\in (-\infty, \alpha_1)$ and has limits $0$ and $-\infty$ at the end points. Besides, $F(\lambda)$ is 
positive on $(0, \infty)$ and its limits at the end points are $\infty$ and $0$. Figure \ref{graph3} summarises the above arguments. 
\begin{figure}[h]
\includegraphics[scale=0.4]{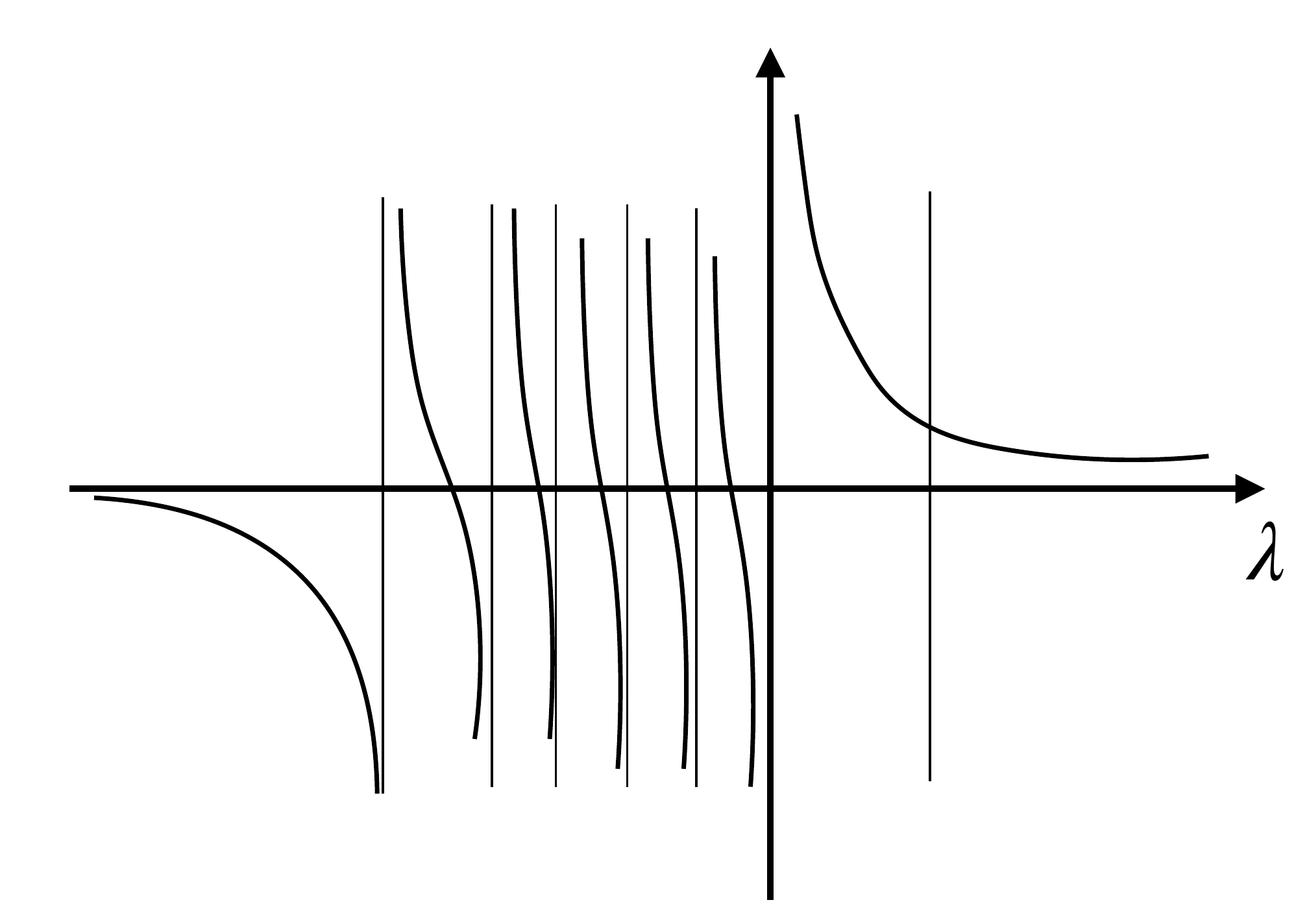}
\caption{The graph of the function $F(\lambda)$.}\label{graph3}
\end{figure}
Thus we see that there is a unique solution of (\ref{eigen1}) in each interval $(\alpha_i, \alpha_{i+1})$; besides, there is a unique solution of (\ref{eigen1}) in the interval $(0, \infty)$ which, as we know, is $\lambda=1$. 

Finally, we note that any of the above solutions automatically satisfies (\ref{eigen2}). Indeed, if $\lambda\in (\alpha_i, \alpha_{i+1})$ then $\alpha_j -\lambda >\alpha_{i+2}-\alpha_{i+1}$ and 
$$\sum_{j\ge i+2} p_j(\lambda - \alpha_j)^{-2} < (\alpha_{i+2}-\alpha_{i+1})^{-2}\cdot \sum_{j\ge i+2} p_j\, <\, \infty.$$
This completes the proof. 
\end{proof}

We may now restate the results of this section for the Laplacian $\Delta=I-P$. 

\begin{proposition}\label{prop:spectrumd} Consider the infinite complete graph $(G, m)$ with $V(G) =\Bbb N$ and weights $m(ij)=p_ip_j$ where $p_1>p_2>\dots >0$ is a sequence with $\sum_{i\ge 1} p_i=1$. 
The spectrum of the Laplacian $\Delta: L^2(G, m)\to L^2(G, m)$ contains the point $1$ as an accumulation point and the other points of the spectrum are simple eigenvalues. The point $0$ is the unique eigenvalue in the interval $[0, 1)$. The remaining eigenvalues $\dots \, <\, \mu_3\, <\,  \mu_2\, <\,  \mu_1\, \le\, 2$ lie 
in $(1, 2)$ and satisfy 
\begin{eqnarray}
q_{i+1}^{-1} <\mu_i< q_i^{-1},\quad  \mbox{where}\quad q_i=1-p_i, \quad \mbox{for}\quad i=1, 2, \dots,
\end{eqnarray}
hence $\mu_i\to 1$. More specifically, for $i=1, 2, \dots,$ each eigenvalue $\mu_i$ is the unique solution of the equation
\begin{eqnarray}\label{equation}
\sum_{j=1}^\infty \frac{p_j}{q_j \mu -1} \, =\, -1
\end{eqnarray}
lying in the interval $(q_{i+1}^{-1}, q_i^{-1})$. 
The spectral gap equals $1$ and top of the spectrum $\mu_{top}=\mu_1$ satisfies
\begin{eqnarray}\label{47}
q_2^{-1}\, < \, \mu_{top}\, <\,  q_1^{-1}. 
\end{eqnarray}
\end{proposition}
\begin{proof}
Since $\Delta = I-P$ we see that the result of Proposition \ref{prop:spectrumd} follows from Lemma \ref{lm:spectrump} by applying the affine transformation $x\mapsto 1-x$. The points of division $\alpha_j$  are mapped into $1-\alpha_j= q_j^{-1}$.
\end{proof}

We can mention another form (\ref{equation4}) of the eigenvalue equation (\ref{equation}) which involves the quantities $r_j=q_j^{-1}$, where $j=1, 2, \dots$
\begin{eqnarray}\label{equation4}
\sum_{j=1}^\infty \frac{r_j-1}{r_j-\mu} =1.
\end{eqnarray}
Note that here $r_j>1$ and $r_j\to 1$. We know that equation (\ref{equation4}) has exactly one solution in each interval 
$(r_{j+1}, r_{j})$ and $\mu=0$ is an additional solution.

\section{Improved estimates for $\mu_1=\mu_{top}$}\label{sec9}

Our goal in this section is to strengthen the inequality (\ref{47}) for the top eigenvalue $\mu_1=\mu_{top}$. Similar method can be applied for more precise estimates of the other eigenvalues. 

Recall that the parameters of the infinite complete graph satisfy $p_1>p_2>\dots$ and $\sum_j p_j =1$. We denote 
$r_j=(1-p_j)^{-1}$. We have $r_1 >r_2>r_3>\dots >1$ and $r_j\to 1$. 
Consider the following quadratic equation 
\begin{eqnarray}\label{equation2b}
\frac{r_1-1}{r_1-\mu}+ \frac{r_2-1}{r_2-\mu} \, =\, 1-x,
\end{eqnarray}
where $x$ is a parameter. 

\begin{lemma} \label{improveb} 
\hfill
\begin{enumerate}
\item[{\rm (A)}] For any value of the parameter $x$ the quadratic equation (\ref{equation2b}) has a unique solution $\mu(x)$ lying in the interval 
$(r_2, r_1)$.  
\item[{\rm (B)}] For $x<1$, the other solution of  the equation (\ref{equation2b}) lies in the interval $(-\infty, r_2)$; for $x>1$,
 the other solution of (\ref{equation2b})
 lies in the interval $(r_1, \infty)$; for $x=1$ the equation (\ref{equation2b}) does not have solutions outside the interval 
 $(r_2, r_1)$. 
\item[{\rm (C)}] $\mu(x)$ is a decreasing function of $x$. 
\item[{\rm (D)}] The top of the spectrum $\mu_{top}=\mu_1$ of the Laplacian $\Delta$ satisfies
\begin{eqnarray}\label{ineqmub}
\mu(x_+) < \mu_{top} < \mu(x_-),
\end{eqnarray}
where $$x_+= \frac{1-p_1-p_2}{1-r_1} \quad \mbox{and}\quad x_-= (1-p_1-p_2)\cdot \frac{r_3}{r_3-r_2}.$$
\end{enumerate}
\end{lemma}

\begin{proof}
The rational function $G(\mu)= \frac{r_1-1}{r_1- \mu}+ \frac{r_2-1}{r_2- \mu}$ has poles at $\mu=r_1$ and $\mu=r_2$, and it is increasing for $\mu\in (r_2, r_1)$. Its graph is shown in Figure \ref{graph5}. 
\begin{figure}[h]
\includegraphics[scale=0.3]{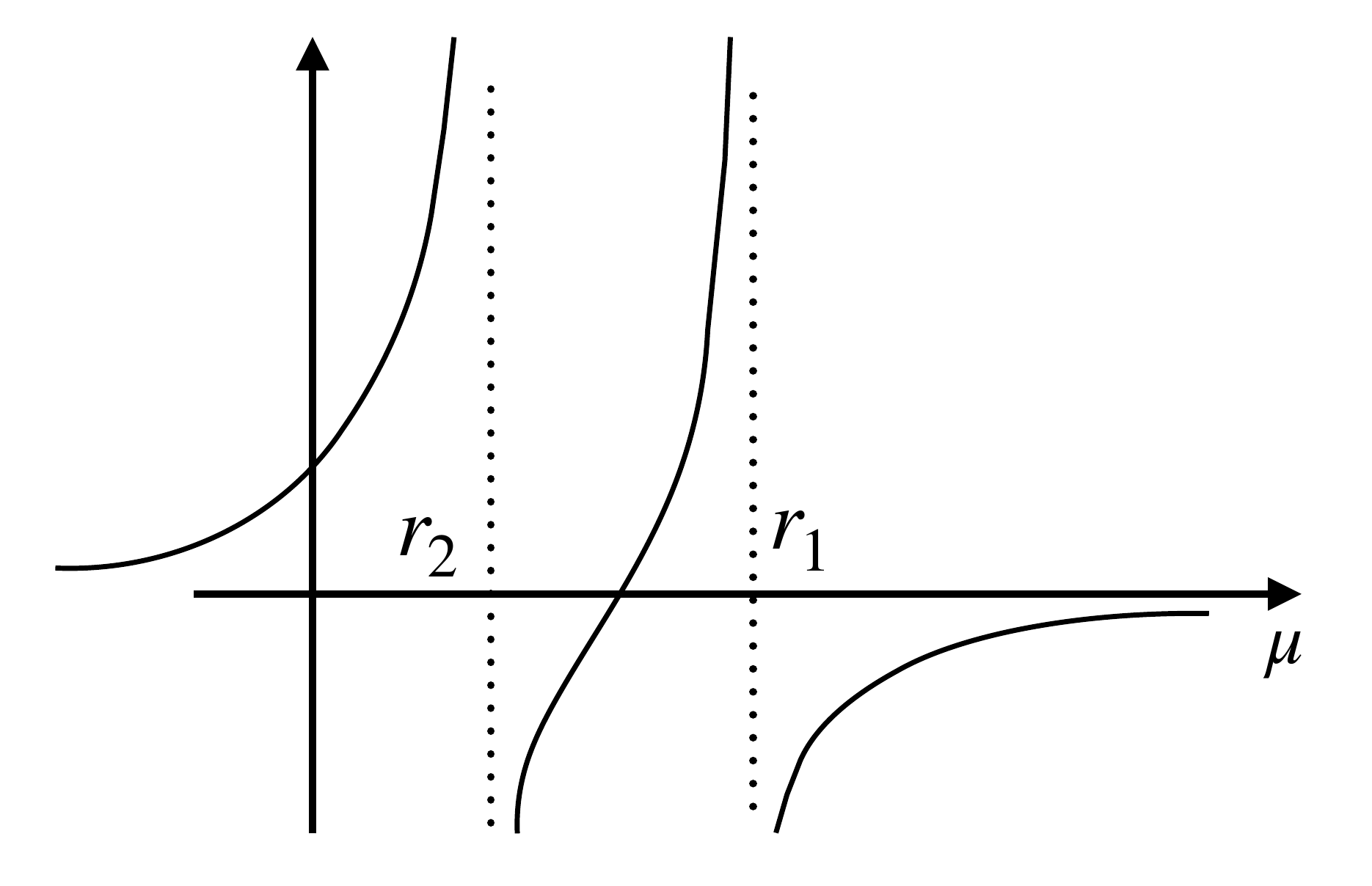}
\caption{The graph of the function $G(\mu)$.}\label{graph5}
\end{figure}
We easily see that $G\colon(r_2, r_1)\to \Bbb R$ is a homeomorphism, which implies (A). Similarly, we see that $G\colon (-\infty, r_2)\to (0, \infty)$ and $G\colon (r_1, \infty)\to (-\infty, 0)$ are monotone increasing homeomorphisms; this implies our statement (B).  

To prove (C), we differentiate the equation $G(\mu(x))=1-x$ getting
$\mu'(x) = - G'(\mu(x))^{-1}.$ This shows that $\mu'<0$ since $G'>0$. 

Finally, to prove (D), we note that in view of equation (\ref{equation4}) one has $\mu_{top}=\mu(x_0)$ with 
$$x_0=\sum_{j\ge 3} \frac{r_j-1}{r_j-\mu}, \quad\mbox{where}\quad \mu=\mu_{top}.$$
We know that $r_2< \mu_{top}< r_1$ and hence 
$$x_0< \sum_{j\ge 3} \frac{r_j-1}{r_j-r_1}= \sum_{j\ge 3} \frac{r_j}{r_j-r_1}\cdot p_j \le \frac{1}{1-r_1}\cdot (1-p_1-p_2) = x_+.$$
Here we used the equalities $\frac{r_j-1}{r_j}=p_j$ and $\sum_{j\ge 1} p_j=1$. Similarly, we have 
$$x_0>\sum_{j\ge 3} \frac{r_j-1}{r_j-r_2}  = \sum_{j\ge 3} \frac{r_j}{r_j-r_2}\cdot p_j \ge \frac{r_3}{r_3-r_2}\cdot 
(1-p_1-p_2)=x_-.$$
Thus, using (C), we conclude that $\mu(x_+)\le \mu_{top}=\mu(x_0)\le \mu(x_-). $
\end{proof}
Below we shall have specific examples illustrating Lemma \ref{improveb}.

\section{Asymmetry of the spectrum and the invariant $\k(G, m)$ for the infinite complete graph}\label{sec10}

In the section we continue studying the infinite complete graph and its spectrum. Our goal is to examine Theorem \ref{thm:mains} in this specific example. 

First we describe the Hausdorff distance $d_H(\sigma(\Delta), \mathcal R(\sigma(\Delta)))$. 

\begin{lemma}\label{asym1} If $\mu_1\le 3/2$ then 
\begin{eqnarray}\label{case1}
d_H(\sigma(\Delta), \mathcal R(\sigma(\Delta))) = 2-\mu_1\,  \ge\,  \frac{1}{2}.
\end{eqnarray}
If $\mu_1\ge 3/2$ then 
\begin{equation}\label{case2}
d_H(\sigma(\Delta), \mathcal R(\sigma(\Delta)))=\frac{1}{2} - \inf_{i\ge 1} \left|\mu_i - \frac{3}{2}  \right| \, \le \, \frac{1}{2}.
\end{equation}
\end{lemma} 
\begin{proof} We know that $\sigma(\Delta) = \{0, 1\}\cup \{\mu_1, \mu_2, \dots \}$ where $\mu_1>\mu_2>\dots >1$, $\mu_i\to 1$.  Hence,  
$\mathcal R(\sigma(\Delta)) = \{1, 2\}\cup \{\bar \mu_1, \bar \mu_2, \dots, \},$ where $\bar \mu_i= 2-\mu_i$. 
We intend to apply Lemma \ref{lm:isometry}. We have: 
$\inf\{\epsilon; \, 2\in \sigma(\Delta)_\epsilon\} = 2-\mu_1$. Besides, $\inf\{\epsilon; \, 1\in \sigma(\Delta)_\epsilon\}=0$ and 
for $i=1, 2, \dots$, one has $\inf\{\epsilon; \, \bar \mu_i\in \sigma(\Delta)_\epsilon\}=\min \{ 2-\mu_i, \mu_i-1\}.$ This clearly implies that, for $\mu_1\le 3/2$, the Hausdorff distance $d_H(\sigma(\Delta), \mathcal R(\sigma(\Delta)))$ equals $2-\mu_1$. 
Consider now the case $\mu_1\ge 3/2$. By the above, the Hausdorff distance $d_H(\sigma(\Delta), \mathcal R(\sigma(\Delta)))$ in this case equals 
$\sup_{i\ge 1} \{ \min\{2-\mu_i, \mu_i-1\}\}.$
We may write 
$$\min\{2-\mu_i, \mu_i-1\}= \frac{1}{2} + \min\{\frac{3}{2} -\mu_i, \mu_i-\frac{3}{2}\}= \frac{1}{2} - \left |\mu_i-\frac{3}{2}\right |.$$
This implies (\ref{case2}). 
\end{proof}
Next we consider our invariant $\k(G, m)$, see Definition \ref{def:kab}. For a subset $A\subset \Bbb N=\V(G)$ we set $P(A) =\sum_{i\in A} p_i$ and $p_{\rm {min}}(A) = \inf_{i\in A}p_i$. Then for a partition $A\sqcup B = \Bbb N$ one has 
$$\k(A,B) = \max\left\{ \frac{P(A)-p_{\rm {min}}(A)}{1-p_{\rm {min}}(A)}, \frac{P(B)-p_{\rm {min}}(B)}{1-p_{\rm {min}}(B)}\right\}.$$
For an infinite set $A\subset \Bbb N=\V(G)$ clearly $p_{\rm {min}}(A)=0$. Hence, if both sets $A, B$ are infinite, we have
$$\k(A, B)=\max\{P(A), P(B)\}\ge 1/2$$ since $P(A)+P(B)=1$. Therefore, the infimum of the numbers $\k(A,B)$ taken over all partitions 
$A\sqcup B=\Bbb N$ with both sets $A, B$ infinite equals 
$$\inf\{P(A); \, A\subset \Bbb N \, \,  \mbox{is infinite and} \,\, P(A)\ge 1/2\}.$$


Next, we  consider partitions $A\sqcup B=\Bbb N$ with $A$ finite and hence $B$ infinite. 
We shall be mainly interested in the situation when $p_1\ge 1/2$. In this case one may take the partition 
$A'=\{1\}$, $B'=\{2, 3, \dots\}$; then $\k(A', B')=P(B')=1-p_1$. 

Let us show that, in fact,  $$\k(G, m)= 1- p_1$$ under the assumption $p_1\ge 1/2$. All partitions with both sets $A$ and $B$ infinite give $\k(A, B)\ge 1/2\ge 1-p_1$. 
Consider a partition $A\sqcup B=\Bbb N$ with finite $A$, $A\not=\{1\}$. If $p_{\rm {min}}(A) =p_i$, $i>1$, then 
$$\k(A, B) = \max\left\{\frac{P(A)-p_i}{1-p_i}, P(B)\right\}.$$
If $1\notin A$ then $1\in B$ and $\k(A, B)= P(B) \ge 1-p_1.$ Consider now the remaining case: $1\in A$ and $p_{\rm {min}}(A)=p_i$ where $i>1$. Then $P(B)\le 1-p_1-p_i<1/2$ and 
$$\frac{P(A)-p_i}{1-p_i}\ge \frac{p_1}{1-p_i}\ge 1/2,$$ and thus $\k(A, B) = \frac{P(A)-p_i}{1-p_i}\ge 1/2 \ge 1-p_1.$
We summarise our above arguments as follows.

\begin{lemma}\label{lm:36}
If $p_1\ge 1/2$ then $\k(G, m)=1-p_1$. 
\end{lemma}

\begin{example}
Consider an infinite complete graph with the following parameters: $p_1=0.9$, $p_2=0.09$, $p_3=0.009$; the values $p_4, p_5, \dots$ will be irrelevant for our estimates below. 

By Lemma \ref{lm:36}, one has $\k(G, m)=0.1$ and applying Theorem \ref{thm:mains} we obtain that the spectral asymmetry 
satisfies $$d_H(\sigma(\Delta), \mathcal R(\sigma(\Delta)))\le 0.2.$$ 
Next, we apply Lemma \ref{asym1}, which gives $\mu_1>3/2$, and the equality (\ref{case2}) implies
$$\inf_{i\ge 1} \left | \mu_i-\frac{3}{2}\right |\ge 0.5- 0.2= 0.3.$$
In other words, the open interval 
$$(1.2, 1.8)=\left(\frac{3}{2}-\frac{3}{10}, \frac{3}{2}+\frac{3}{10}\right)\subset [0,2]$$ 
contains no eigenvalues. Thus, our inequality (\ref{Hausdorff}) allows finding lacunas in the spectrum. 

We can also estimate the Hausdorff distance $d_H(\sigma(\Delta), \mathcal R(\sigma(\Delta)))$ using information about the eigenvalues. 
We have $r_1= q_1^{-1} = 10$, $r_2= q_2^{-1}=\frac{100}{91}$ and hence $1.099< \mu_1< 10$, by the inequality (\ref{47}), which is not informative enough.  We can improve the bound (\ref{47}) using Lemma \ref{improveb}. We obtain
$x_+= - \frac{0.01}{9}\approx - 0.001111$, and $x_-= -\frac{0.01}{\frac{0.991}{0.91}-1}\approx - 0.11235$. Solving the quaratic equation (\ref{equation2b}) we obtain
$1.94747 \le \mu_1\le 2.4194,$ which, of course, should be understood as
$$1.94747 \le \mu_1\le 2.$$
Next we estimate $\mu_2$ using (\ref{47}). We obtain 
$$1.00908\approx \frac{1000}{991}= r_3 =q_3^{-1} \, <\, \mu_2\, <\,  r_2= q_2^{-1} =\frac{100}{91}\approx 1.099.$$
Thus, using Lemma \ref{asym1}, we obtain for the Hausdorff distance
$$0.00908 \le d_H(\sigma(\Delta), \mathcal R(\sigma(\Delta)))\le 0.099.$$
\end{example} 

\begin{remark}
It is an interesting inverse problem whether the spectrum of the infinite complete graph determines the sequence of shape parameters $p_1, p_2, \dots, $ satisfying $\sum_{i\ge 1} p_i=1$. 
\end{remark}

 \bibliographystyle{amsplain}

\end{document}